\documentclass{article}
\usepackage{style}
\usepackage[utf8]{inputenc}

\begin{document}


\title{Average-case and smoothed analysis of graph isomorphism}
\author{
    Julia Gaudio\thanks{Northwestern University; \url{julia.gaudio@northwestern.edu}.}
    \and
    Mikl\'os Z.\ R\'acz\thanks{Northwestern University; \url{miklos.racz@northwestern.edu}.} 
    \and
    Anirudh Sridhar\thanks{MIT; \url{anisri@mit.edu}.} 
}
\date{\today}

\maketitle


\begin{abstract}
We propose a simple and efficient local algorithm for graph isomorphism which succeeds for a large class of sparse graphs. This algorithm produces a low-depth canonical labeling, which is a labeling of the vertices of the graph that identifies its isomorphism class using vertices' local neighborhoods.

Prior work by Czajka and Pandurangan showed that the degree profile of a vertex (i.e., the sorted list of the degrees of its neighbors) gives a canonical labeling with high probability when $n p_n = \omega( \log^{4}(n) / \log \log n )$ (and $p_{n} \leq 1/2$); subsequently, Mossel and Ross showed that the same holds when $n p_n = \omega( \log^{2}(n) )$. We first show that their analysis essentially cannot be improved: we prove that when $n p_n = o( \log^{2}(n) / (\log \log n)^{3} )$, with high probability there exist distinct vertices with isomorphic $2$-neighborhoods. Our first main result is a positive counterpart to this, showing that $3$-neighborhoods give a canonical labeling when $n p_n \geq (1+\delta) \log n$ (and $p_n \leq 1/2$); 
this improves a recent result of Ding, Ma, Wu, and Xu, completing the picture above the connectivity threshold. 

Our second main result is a smoothed analysis of graph isomorphism, showing that for a large class of deterministic graphs, a small random perturbation ensures that $3$-neighborhoods give a canonical labeling with high probability. 
While the worst-case complexity of graph isomorphism is still unknown, this shows that graph isomorphism has polynomial smoothed complexity.
\end{abstract}


\section{Introduction} \label{sec:intro} 
Graph isomorphism is a fundamental problem in computer science. Given a pair of graphs $G$ and $H$ on $n$ vertices, the graph isomorphism problem is to decide whether $G$ and $H$ are isomorphic. Curiously, graph isomorphism has neither been shown to be polynomial-time solvable, nor to be NP-complete. Assuming $\mathrm{P} \neq \mathrm{NP}$, the problem is believed to lie in the so-called NP-intermediate complexity class~\cite{Schoning1988}. 
Recently, Babai gave a quasi-polynomial time algorithm for graph isomorphism~\cite{Babai2016} (see also the expository papers~\cite{Helfgott2017,babai2018group}).

The focus of this paper is on \emph{canonical labeling} algorithms for graph isomorphism.
A canonical labeling of a graph is an assignment of a unique label to each vertex which identifies its isomorphism class. More precisely, a labeling, denoted $f(\cdot)$, is a mapping on graphs. Given a graph $G$, the labeling $f(G)$ produces a graph which is isomorphic to $G$. A labeling is canonical if the following property is satisfied: whenever $G$ and $H$ are isomorphic, then the graphs $f(G)$ and $f(H)$ are identical. It is well-known that such a labeling can be used to solve the graph isomorphism problem~\cite{babai1979canonical, babai1980random, babai1983canonical, karp1979probabilistic, lipton1978beacon, babai1980complexity}, 
which has motivated the study of canonical labeling algorithms for certain classes of graphs~\cite{babai1983canonical, lipton1978beacon, babai1980complexity} and for random graphs~\cite{babai1980random, babai1979canonical, karp1979probabilistic, lipton1978beacon}.

The study of canonical labeling algorithms for random graphs was initiated by Babai, Erd\H{o}s, and Selkow~\cite{babai1980random}, who gave a simple $O(n^{2})$ time canonical labeling algorithm for $G(n,1/2)$ which succeeds with high probability\footnote{Throughout the paper we say that an event holds with high probability if it holds with probability tending to $1$ as $n\to\infty$.}. Following (the circulation of) \cite{babai1980random}, several works strengthened this result~\cite{karp1979probabilistic,lipton1978beacon,babai1979canonical}. 
Since $G(n,1/2)$ can be viewed as a uniformly random sample from the set of all (labeled) graphs on $n$ vertices, this result implies that almost all graphs admit an efficiently computable canonical labeling. 
Our paper contributes to the understanding of efficient canonical labeling in two ways:
\begin{enumerate}
    \item We show that \emph{sparse} graphs admit an efficiently computable canonical labeling with high probability. In parallel to the line of work on $G(n,1/2)$, we study canonical labeling of $G(n,p_n)$ where $p_n = o(1)$, 
    and give an efficient local canonical labeling algorithm above the connectivity threshold $p_{n} \geq (1+\delta) \log(n) / n$. 
    \item Taking a smoothed analysis perspective, we consider canonical labeling of randomly perturbed deterministic graphs. In particular, we show that for every \emph{deterministic} graph among a large class of graphs (which includes all graphs with degree bounded by $n^{c}$ for $c < 1/2$), a sparse random perturbation ensures an efficiently computable canonical labeling.
\end{enumerate}

\subsection{Canonical labeling of sparse random graphs} 
The study of canonical labeling of sparse random graphs was initiated by 
Bollob\'as~\cite[Theorem 3.17]{bollobas2001random}, 
showing that 
an extension of the algorithm in~\cite{babai1980random} provides a canonical labeling of $G(n,p_n)$ with high probability when  
$p_{n} = \omega( \log(n) / n^{1/5})$ 
and $p_{n} \leq 1/2$. 
In a separate work~\cite{Bollobas1982}, Bollob\'as gave an efficient, high probability canonical labeling algorithm when 
$p_{n} \geq (\log(n) + \psi_{n})/ n$, 
where $\psi_{n} \to \infty$, 
and 
$p_{n} \leq 2n^{-11/12}$. 
Note that this algorithm works all the way down to the connectivity threshold of $G(n, p_n)$. 
Canonical labeling is also possible below the connectivity threshold; Linial and Moshieff \cite{Linial2017} showed that it is efficiently achievable with high probability when $p_n = \omega(1/n)$ and $p_n \leq 1/2$. When $p_n = o(1/n)$, canonical labeling is also possible since $G(n,p_n)$ is a collection of trees with high probability, and trees admit an efficient canonical labeling algorithm \cite{Read1972}. 

The aforementioned canonical labeling algorithms \cite{babai1980random,karp1979probabilistic,lipton1978beacon,babai1979canonical, bollobas2001random, Bollobas1982} crucially leverage \emph{global} information about the graph (e.g., locations of high degree vertices). In contrast, the focus of our work is on \emph{local} canonical labeling algorithms, where the label of each vertex depends only on its low-depth neighborhood. Our focus is on the sparse regime where $p_n = n^{-1 + o(1)}$ (though some of our results hold more broadly). 
Czajka and Pandurangan~\cite{Czajka2008} developed a high probability canonical labeling algorithm from 2-neighborhoods\footnote{For any $r \in \mathbb{N}$, the $r$-neighborhood of $v$ in $G$ is the induced subgraph on all vertices at distance at most $r$ from $v$.} in the regime where $n p_{n} = \omega ( \log^{4}(n) / \log \log n )$ 
and $p_{n} \leq 1/2$. 
Their algorithm is simple and effective: 
to each vertex $v$ it assigns 
the sorted list of the degrees of $v$'s neighbors. 
Mossel and Ross~\cite{Mossel2019} improved the analysis of the same algorithm, 
showing that it works whenever 
$n p_{n} = \omega( \log^{2}(n) )$ 
and $p_{n} \leq 1/2$.

Our first result identifies a regime in which depth-$2$ neighborhoods are not unique,
showing that the analysis of~\cite{Mossel2019} essentially cannot be improved further. 

\begin{theorem}\label{thm:2nbrhood_impossibility}
Let $G \sim G(n,p_{n})$ and assume that $n p_{n} = o ( \log^{2}(n) / \left( \log \log n \right)^{3})$. 
Then, with high probability, 
there exist distinct vertices $u, v \in V(G)$ such that 
the $2$-neighborhoods of $u$ and $v$ are isomorphic. 
\end{theorem}

Recent independent and concurrent work by Johnston, Kronenberg, Roberts, and Scott~\cite{johnston2022shotgun} also proves Theorem~\ref{thm:2nbrhood_impossibility} and shows that the loglog factors are indeed necessary.
While non-uniqueness of $2$-neighborhoods does not rule out the possibility of a depth-$2$ canonical labeling algorithm, it does rule out the possibility of using the isomorphism class of each vertex's $2$-neighborhood as a canonical label. Theorem~\ref{thm:2nbrhood_impossibility} naturally raises the question of whether there exists a local canonical labeling algorithm in this regime using a somewhat larger local neighborhood.\footnote{For $p_n \geq \sqrt{\frac{(2+\epsilon) \log n}{n}}$, the graph $G \sim G(n,p_n)$ has diameter $2$ with high probability~\cite{Bollobas1981}, so $2$-neighborhoods include the whole graph. However, for sufficiently sparser graphs, $2$-neighborhoods and $3$-neighborhoods are truly local.}
In particular, do $3$-neighborhoods provide a canonical labeling in this regime? 
Recent work of Ding, Ma, Wu, and Xu \cite{Ding2021} 
showed that this is the case when $np_n \geq C\log n$ for a sufficiently large constant~$C$.
Our next result completes the picture, showing that $3$-neighborhoods are sufficient (essentially) all the way down to the connectivity threshold.
\begin{theorem}\label{thm:3nbrhood_algorithm}
Fix $\delta > 0$ and let $p_{n}$ satisfy 
$n p_{n} \geq (1+\delta) \log(n)$ 
and $p_{n} \leq \frac{1}{2}$. 
Let $G \sim G(n,p_{n})$. 
Then, with high probability, 
for every pair of distinct vertices $u, v \in V(G)$, 
the $3$-neighborhoods of $u$ and $v$ are nonisomorphic. 
\end{theorem}
Observe that if $np_n \leq (1-\delta) \log (n)$ for some $\delta > 0$, then the $3$-neighborhoods are not unique due to the existence of isolated vertices (though this does not rule out the possibility of a depth-$3$ canonical labeling). This observation along with Theorem \ref{thm:3nbrhood_algorithm} establishes $p_n = \frac{\log n}{n}$ as the threshold function for uniqueness of $3$-neighborhoods. Specifically, let $G_n \sim G(n,p_n)$ where $p_n = \frac{a \log n}{n}$. Then
\begin{align*}
\lim_{n \to \infty} \mathbb{P}(G_n\text{ has unique } 3\text{-neighborhoods}) &=
\begin{cases}
1 & \text{if } a > 1\\
0 & \text{if } a < 1.
\end{cases}
\end{align*}
Combining 
Theorems~\ref{thm:2nbrhood_impossibility} and~\ref{thm:3nbrhood_algorithm} 
shows that a depth of $3$ is necessary and sufficient in order to ensure unique local neighborhoods, when $(1+\delta) \log(n) \leq np_n = o ( \log^2 (n) / (\log \log n)^{3})$.

Our canonical labeling algorithm is closely related to the Color Refinement (CR) algorithm\footnote{The Color Refinement Algorithm is also referred to as the $1$-dimensional Weisfeiler--Leman Algorithm~\cite{Leman1968}. In turn, the $k$-dimensional Weisfeiler--Leman Algorithm is a generalization of the CR algorithm, which colors $k$-tuples of vertices.}, which was proposed in numerous works and attributed to Morgan~\cite{Morgan1965}. 

The CR method is an iterative labeling scheme: first, every vertex is labeled according to its degree. Next, every vertex is labeled according to the multiset of its neighbors' degrees. The process continues until every vertex has a unique label. Babai and Ku\v{c}era \cite{babai1979canonical} proposed a lossy version of the CR method, in which the initial labels are the degrees and subsequent labels are computed modulo $4$, showing that three steps of this refinement procedure succeed with high probability on $G \sim G(n,1/2)$.
Our canonical labeling algorithm can also be seen as implementing three levels of the CR scheme. However, instead of labeling vertices according to their degrees in the first step, we label vertices according to their degrees, modulo a sufficiently large constant $m$. The initial labels of the vertices thus behave as though they were produced from a uniformly random coloring with $m$ labels, an insight that will be crucial to the analysis. By addressing the case of sparse random graphs, this paper expands the repertoire of CR-style methods. 

Color Refinement can be implemented in $O((n+|E|) \log n)$ time, where $|E|$ is the number of edges in the graph~\cite{Grohe2020,Cardon1982}. Our algorithm, which additionally computes degrees modulo a constant $m$, also has this runtime. In the regime under consideration, the runtime is $O(n^2p_n \log n)$ with high probability. 
When $np_n = \Theta( \log n)$, this is significantly more efficient than the \emph{non-local} canonical labeling algorithm of Bollob\'{a}s \cite{Bollobas1982}. Indeed, in this regime the algorithm of \cite{Bollobas1982} requires determining the neighbors of a given vertex up to distance $r_n = \Theta ( \sqrt{ \log(n) / \log (np_n) } ) = \omega(1)$, leading to a time complexity of $O( n \Delta^{r_n})$.\footnote{A repeated application of the single-source shortest path algorithm of Thorup \cite{Thorup1999} yields a runtime of $O(n|E|)$, a looser bound than $O(n \Delta^{r_n})$.} 

\subsection{Smoothed analysis for canonical labeling}\label{sec:smoothed_results}
The techniques that we develop for Theorem~\ref{thm:3nbrhood_algorithm} above allow us to go beyond the average-case setting to the smoothed analysis setting. 
Informally, our second result shows that for any graph $G$ from a large class of graphs, a random perturbation of $G$ admits a depth-$3$ canonical labeling with high probability. To formalize this result, we need several definitions. First, given a graph $G = (V,E)$, a vertex $v \in V$, and $r \geq 1$, let $\mathcal{N}_r(v; G)$ denote the $r$-neighborhood of $v$ in the graph $G$. 
We consider the following large class of graphs, where $0 <\lambda < 1$ is an arbitrary constant:
\begin{equation}
\mathcal{G}_n(\lambda) := \left\{G = ([n], E) : |\mathcal{N}_2(v; G)| \leq n^{\lambda}, \forall v \in[n]\right\}.  \label{eq:deterministic-class}
\end{equation}
Here $|\mathcal{N}_2(v; G)|$ is the number vertices in the $2$-neighborhood of $v \in [n]$. 
In words, $\mathcal{G}_{n}(\lambda)$ contains all graphs where every $2$-neighborhood has at most $n^{\lambda}$ vertices. 
Note that for constant $\delta > 0$ and $n$ large enough, any graph $G$ with degree bounded by $n^{\frac{1}{2} - \delta}$ belongs to $\mathcal{G}_n(1-\delta)$. Also, note that $\mathcal{G}_n(\lambda) \subseteq \mathcal{G}_n(\lambda')$ when $\lambda < \lambda'$.
Our main result is the following:
\begin{theorem}\label{theorem:smoothed-analysis}
Fix $0 < \lambda <1$. Let $G_1 \in \mathcal{G}_n(\lambda)$ be arbitrary, 
and let $G_2 \sim G(n,p_n)$, 
where 
$n p_n = \omega\left(\log^{2} (n) \right)$ 
and 
$n p_n = o\left(n / \log^3(n) \right)$.
\begin{enumerate}
    \item Let $G$ be the union of $G_1$ and $G_2$ (that is, $(u,v)$ is an edge in $G$ if it is an edge in either~$G_1$ or~$G_2$, or both). Then, with high probability, all $3$-neighborhoods of $G$ are nonisomorphic.
    \item Let $G'$ be the XOR of $G_1$ and $G_2$ (that is, $(u,v)$ is an edge in $G'$ if it is an edge in exactly one of $G_1$ and $G_2$). Then, with high probability, all $3$-neighborhoods of $G'$ are nonisomorphic. 
\end{enumerate}
In both models, it follows that a canonical labeling can be found in $O((n+ |E|) \log n)$ time, where $|E|$ is the number of edges in $G$.
\end{theorem}
The first statement shows that ``sprinkling'' additional edges on $G_1$ ensures uniqueness of $3$-neighborhoods, while the second statement shows that randomly flipping each edge in $G_1$ ensures the same. Both perturbation models can be viewed as a smoothed model of graphs, allowing for a smoothed analysis of canonical labeling. While these models are qualitatively different from the random graph model considered in Theorem \ref{thm:3nbrhood_algorithm}, the analysis is quite similar, crucially relying on the random coloring perspective that we introduce. In fact, the restriction in \eqref{eq:deterministic-class} is to ensure that the color labels behave randomly. 

Theorem~\ref{theorem:smoothed-analysis} shows the existence of an efficient local canonical labeling algorithm in the smoothed analysis setting, which implies that graph isomorphism has polynomial smoothed complexity under mild conditions. 
At the same time, the result above opens up a host of interesting technical questions towards a better understanding of the smoothed complexity of graph isomorphism. 
Can the random perturbation $G_2$ be made sparser? In particular, is it sufficient for $G_2$ to be above the connectivity threshold (i.e., can the condition $n p_n = \omega\left(\log^{2} (n) \right)$ be relaxed to $n p_n \geq (1+\delta) \log (n)$ for any $\delta > 0$)? 
Can the graph class $\mathcal{G}_{n}(\lambda)$ be further enlarged? 
We conjecture that these questions have affirmative answers. 
A further natural question is whether $r$-neighborhoods for constant $r > 3$ can help for canonical labeling algorithms in certain regimes of the smoothed analysis setting.

\subsection{Implications and connections to related problems}
Our results lead to several implications.
First, Theorem \ref{theorem:smoothed-analysis} implies that any graph in $\mathcal{G}_n(\lambda)$ is close (in terms of edge Hamming distance, denoted $d_H(G,G')$) to some graph which has unique $3$-neighborhoods. 
\begin{corollary}\label{cor:hamming}
Let $0 < \lambda < 1$ and $\delta, \epsilon > 0$. Then for any $G_1 \in \mathcal{G}_n(\lambda)$, 
there exists a graph $G'$ with unique $3$-neighborhoods, such that 
\begin{equation}
(1-\epsilon)\frac{\log^{2+\delta}(n)}{n} \binom{n}{2} \leq d_H(G,G') \leq (1+\epsilon)\frac{\log^{2+\delta}(n)}{n} \binom{n}{2}.   \label{eq:Hamming}
\end{equation}
\end{corollary}
\begin{proof}
Using the Probabilistic Method, we show that such a graph $G'$ exists. Let $G_2 \sim G(n,p_n)$, where $p_n = \frac{\log^{2+\delta} (n)}{n}$, and let $G'$ be the XOR of $G_1$ and $G_2$. By Theorem \ref{theorem:smoothed-analysis} the graph $G'$ has unique $3$-neighborhoods, with high probability. Furthermore,  with high probability, the number of edges in $G_2$ falls into the set $\left[(1-\epsilon)\frac{\log^{2+\delta}n}{n} \binom{n}{2},\dots, (1+\epsilon)\frac{\log^{2+\delta}n}{n} \binom{n}{2} \right]$, so that $G'$ satisfies \eqref{eq:Hamming} with high probability. 
\end{proof}
An interesting question is whether the Hamming radius in \eqref{eq:Hamming} can be reduced; does every deterministic graph lie even closer to some graph with unique $3$-neighborhoods?

Second, Theorem \ref{thm:3nbrhood_algorithm} yields an efficient algorithm for matching two isomorphic random graphs.
\begin{corollary}\label{corollary:isomorphism}
Let $\pi : [n] \to [n]$ be an arbitrary permutation. Let $p_{n}$ satisfy 
$n p_{n} \geq (1+\delta) \log(n)$ 
and $p_{n} \leq 1/2$. Let $G_1 \sim G(n,p_n)$. Let $G_2$ be the graph where $(\pi(i), \pi(j))$ is an edge in $G_2$ if and only if $(i,j)$ is an edge in $G_1$. Let $\Delta := \max_{i \in [n]} \deg_{G_1}(i)$.
Then there is an algorithm which matches the vertices of $G_1$ and $G_2$ which succeeds with high probability and runs in time $O(n \Delta \log n)$.
\end{corollary}

An analogous result holds also in the smoothed analysis setting. 
\begin{corollary}\label{corollary:isomorphism_smoothed}
Let $\pi : [n] \to [n]$ be an arbitrary permutation. 
Fix $0 < \lambda < 1$ 
and let $p_{n}$ satisfy 
$n p_n = \omega\left(\log^{2} (n) \right)$ 
and 
$n p_n = o\left(n / \log^3(n) \right)$. 
Let $G_{1}$ be generated in either one of the two ways that $G$ is generated in the setting of Theorem~\ref{theorem:smoothed-analysis}. 
Let $G_2$ be the graph where $(\pi(i), \pi(j))$ is an edge in $G_2$ if and only if $(i,j)$ is an edge in $G_1$. Let $\Delta := \max_{i \in [n]} \deg_{G_1}(i)$.
Then there is an algorithm which matches the vertices of $G_1$ and $G_2$ which succeeds with high probability and runs in time $O(n \Delta \log n)$.
\end{corollary}

Theorem \ref{thm:3nbrhood_algorithm} also implies a result on \emph{graph shotgun assembly}. The graph shotgun assembly problem, introduced by Mossel and Ross \cite{Mossel2019} and subsequently studied by several authors \cite{gaudio2022shotgun, yartseva2016assembling, Huang2021, adhikari2022shotgun, ding2022shotgun, mossel2015shotgun}, asks whether a graph can be reconstructed from the collection of local neighborhoods. Specifically, for each vertex $v \in [n]$, we are given the subgraph of $G$ induced by the vertices which are at distance at most $r$ from $v$. The subgraph is unlabeled apart from identifying the central vertex $v$ (and typically we can find the center even if it is not given). A graph $G$ is \emph{reconstructable} if the only graphs $H$ which have the same collection of unlabeled local neighborhoods  are those isomorphic to $G$. Mossel and Ross \cite{Mossel2019} raised the question of what neighborhood radius is sufficient for shotgun assembly, when $n p_n = \omega(1)$ and $n p_n = O(\log^{2}(n))$. The following result gives a partial answer.
\begin{corollary}\label{corollary:shotgun}
Let $G \sim G(n,p_n)$, where $n p_{n} \geq (1+\delta) \log(n)$. Then, with high probability, $G$ is reconstructable at radius $r = 4$.
\end{corollary}
\begin{proof}
When $p_n \geq 1/2$, the graph $G$ has diameter $2$ with high probability, so the conclusion follows. 
Now assume that $n p_{n} \geq (1+\delta) \log(n)$ and $p_{n} \leq 1/2$. 
By \cite[Lemma 2.4]{Mossel2019}, a graph with non-isomorphic $(r-1)$-neighborhoods is reconstructable at radius~$r$. The claim follows from the uniqueness of $3$-neighborhoods, guaranteed by Theorem~\ref{thm:3nbrhood_algorithm}.
\end{proof}
In independent and concurrent work,  Johnston~et~al.~\cite[Theorem 1.7 (iii)]{johnston2022shotgun} 
show a similar result, which requires the slightly stronger condition of $p_n \geq \nicefrac{\beta \log (n)}{n}$ (where $\beta > 30$ suffices).

Our work on canonical labeling is also intimately connected to the \emph{graph matching} problem \cite{Pedarsani2011}, which can be viewed as a noisy version of the graph isomorphism problem. Specifically, two \emph{correlated} graphs $(G_1, G_2)$ are generated as follows: first, a ``parent'' Erd\H{o}s-R\'{e}nyi graph $G_0 \sim G(n, p_n)$ is constructed, and two graphs $(G_1, G_2')$ are generated by independently subsampling edges in $G_0$ with some probability $s \in [0,1]$. Finally, $G_2$ is generated by permuting the vertex labels of $G_2'$ according to a uniform random permutation $\pi:[n] \to [n]$. The objective of graph matching is to recover the latent permutation $\pi$. Notice that when $s = 1$, graph matching is equivalent to matching two isomorphic graphs, and Corollary~\ref{corollary:isomorphism} applies.

In recent years, local canonical labeling algorithms for graph matching have led to \emph{efficient algorithms} for this task \cite{dai2019analysis,Ding2021, mao2021random, mao2022exact, mao2022otter, ganassali2020tree, gml2022correlation, piccioli2022aligning}. At a high level, such algorithms have two main steps: (1) a (local) labeling is constructed and (2) vertices whose labels are ``close enough'' are matched. Clearly, the success of these algorithms depends crucially on how the labels are constructed. The closest work in this line to ours is that of Ding, Ma, Wu, and Xu~\cite{Ding2021}. They study the effectiveness of 2-neighborhood and 3-neighborhood canonical labels for graph matching, showing that when $1 - s = O ( 1/ \poly \log (n) )$, exact recovery of the latent matching $\pi$ is possible when $np_n = \omega ( \log^2 n)$ with 2-neighborhood canonical labels and possible when $np_n \geq C  \log n$ with 3-neighborhood canonical labels, for $C$ sufficiently large. (See also~\cite{mao2021random} for an analysis of 3-neighborhood canonical labels for smaller $s$ but larger $p_n$.) Significantly, Theorem~\ref{thm:3nbrhood_algorithm} improves upon~\cite{Ding2021, mao2021random} when $s = 1$; an interesting future line of work is to extend our techniques to the graph matching problem to understand potential further improvements when $s < 1$.

We remark that our work has connections to the broader theme of local algorithms on graphs (i.e., those that leverage only local neighborhoods for algorithmic tasks); see \cite{Suomela2013} for a survey. The setting of this body of work is somewhat different from ours, however: the underlying graph is assumed to have a \emph{generic, deterministic} structure, and local algorithms are assumed to have access to fully or partially known vertex labels. In particular, the availability of vertex labels is crucial and often necessary for the design of local algorithms; see, for instance,~\cite{Angluin1980local, Hasemann2016local, Goos2013local, Fraigniaud2018local, Suomela2013}. Fortunately, as shown in Theorem~\ref{thm:3nbrhood_algorithm}, unique vertex labels can be efficiently computed in a local manner in Erd\H{o}s-R\'{e}nyi graphs. A natural follow-up question is whether the assumptions on vertex identifiability can be relaxed -- or even removed -- for local algorithms on random graphs. While we do not pursue this question in the present work, we expect it to be a fascinating direction of future study.

Our results described in Section~\ref{sec:smoothed_results} fall into the general framework of \emph{smoothed analysis}. This framework, pioneered by Spielman and Teng~\cite{spielman2001smoothed,spielman2004smoothed} for the simplex algorithm, continuously interpolates between worst-case and average-case instances, and the resulting \emph{smoothed complexity} of an algorithm is often a better indicator of its performance in practice than its worst-case (or average-case) complexity.  
In smoothed analysis, an arbitrary deterministic input is randomly perturbed via a small random perturbation. 
Such smoothed models arise naturally in many discrete settings, 
such as CSPs~\cite{feige2007refuting,guruswami2022algorithms} (where the negation patterns of the literals are randomized with some small probability), 
graphs~\cite{bohman2003many,krivelevich2006smoothed,makarychev2014constant,krivelevich2017bounded} (where the edge variables are randomly perturbed), and 
problems involving noisy channels like trace reconstruction~\cite{chen2021polynomial} (where the input bits are independently flipped with some small probability). 
While the graph isomorphism problem has received significant attention in both the worst-case (though its complexity remains unknown) and the average-case (see above for discussion), 
to the best of our knowledge our work is the first to consider graph isomorphism in the smoothed analysis setting. 
Specifically, 
Theorem~\ref{theorem:smoothed-analysis} shows that $3$-neighborhoods give an efficient local canonical labeling algorithm in the smoothed analysis setting, 
implying that graph isomorphism has polynomial smoothed complexity (under mild conditions). 

\subsubsection{Comparison to the work of Johnston, Kronenberg, Roberts, and Scott}

Finally, we compare our work to the independent and concurrent work of Johnston~et~al.~\cite{johnston2022shotgun}. 
While there is overlap in technical details between our work and~\cite{johnston2022shotgun}, the underlying motivation for deriving these technical results is different. Specifically, our motivation comes from the graph isomorphism problem, while~\cite{johnston2022shotgun} is motivated by the shotgun assembly problem for graphs. We view it as a strength of these works that the results apply to several different applications, as explained above. 

At a technical level, our main theorems, Theorems~\ref{thm:2nbrhood_impossibility},~\ref{thm:3nbrhood_algorithm}, and~\ref{theorem:smoothed-analysis}, 
give precise results for $2$-neighborhoods and $3$-neighborhoods being distinct (with high probability) in random graphs and randomly perturbed graphs. Theorem~\ref{thm:2nbrhood_impossibility} was also proven in~\cite{johnston2022shotgun}, where a near-matching lower bound was also proven, showing that the loglog factors are necessary. 
The work~\cite{johnston2022shotgun} proves a weaker variant of Theorem~\ref{thm:3nbrhood_algorithm}, showing that the $3$-neighborhoods of not-too-small-degree vertices are distinct, as opposed to Theorem~\ref{thm:3nbrhood_algorithm}, which shows that \emph{all} $3$-neighborhoods are distinct. We also note that the proofs here differ substantially in their details. 
The authors of~\cite{johnston2022shotgun} proved such a claim en route to proving a result about $r$-reconstructability for~$r \geq 4$. Due to Theorem~\ref{thm:3nbrhood_algorithm}, we immediately have a stronger result for $4$-reconstructability, see Corollary~\ref{corollary:shotgun} and the subsequent comparison. 
We note that~\cite{johnston2022shotgun} also contains results on graph shotgun assembly that fall outside of our results.

The results and techniques that we develop do not only apply to random graphs, they also apply to randomly perturbed deterministic graphs under mild conditions, see Theorem~\ref{theorem:smoothed-analysis} and Corollaries~\ref{cor:hamming},~\ref{corollary:isomorphism}, and~\ref{corollary:isomorphism_smoothed} above. 
The addition of these results is the main difference between the initial version of this paper~\cite{GRS22b} and the current version. 
As such, our work goes beyond the purview of~\cite{johnston2022shotgun}, which focuses on random graphs, and justifies the effectiveness of our methods in more practical settings beyond random graphs. 

\subsection{Notation and outline}
We denote the vertex set of $G(n,p_n)$ by $[n]:=\{1,2,\ldots,n\}$, and we use $p_n$ and $p$ interchangeably. We write $u \sim v$ to indicate that the edge $(u,v)$ is present in $G$. Recall that $\mathcal{N}_r(u; G)$ denotes the $r$-neighborhood of a vertex $u \in [n]$ within a graph $G$. When the underlying graph $G$ is clear from context, we write $\mathcal{N}_r(u)$.  
Also let $\mathcal{N}(u) \equiv \mathcal{N}_1(u)$. 
Though $\mathcal{N}_r(u)$ is a graph, we write $j \in \mathcal{N}_r(u)$ to mean that $j$ is an element of the vertex set of $\mathcal{N}_r(u)$. Since we work with subgraphs of $G$, let $G\{A\}$ denote the induced subgraph of $G$ on the vertices $A \subset [n]$. Let $\deg_H(v)$ be the degree of a vertex $v$ with respect to a given graph $H$; when the graph is clear from context, we write simply $\deg(v)$. We use $\mathbf{1} (\cE)$ for the indicator function of an event~$\cE$. We use Bachmann--Landau notation $o(\cdot), O(\cdot), \omega(\cdot), \Omega(\cdot), \Theta(\cdot)$.

The rest of the paper is structured as follows. We provide proof outlines in Section \ref{sec:outlines}. After some preliminaries in Section~\ref{sec:preliminaries}, 
we prove Theorem~\ref{thm:3nbrhood_algorithm} and Corollary~\ref{corollary:isomorphism} in Section~\ref{sec:3nbrhood}. 
We prove Theorem~\ref{thm:2nbrhood_impossibility} in Section~\ref{sec:2nbrhood}. Finally, we prove Theorem \ref{theorem:smoothed-analysis} in Section \ref{sec:smoothed-analysis}.

\section{Proof ideas}\label{sec:outlines}
\subsection{Canonical labeling from $3$-neighborhoods}
\label{subsec:outline_canonical_labeling}
In order to prove Theorem \ref{thm:3nbrhood_algorithm}, we will define a labeling algorithm and show that no two vertices receive the same label, with high probability. Our strategy is similar in spirit to the argument of Czajka and Pandurangan~\cite{Czajka2008}, who proposed a depth-$2$ labeling algorithm. Their algorithm computes \emph{degree profiles}: for a given vertex, its degree profile is the multiset of its neighbors' degrees. Their argument begins by fixing two vertices $\{u,v\}$ and sorting the remaining vertices into ``bins'' according to their degrees. Each bin contains vertices with similar degrees, all falling into a prescribed interval. Observe that if $u$ and $v$ have the same degree profile, then for any bin $B$, $u$ and $v$ must have the same number of neighbors in~$B$. But the number of neighbors that $u$ has to a bin $B$ is binomial, and independent of the number of neighbors that $v$ has to~$B$. Furthermore, the number of neighbors to the other bins are also independent binomials. Using standard results on binomial random variables, it is straightforward to argue that $u$ and $v$ are unlikely to have the same degree profile. Taking a union bound over pairs of distinct vertices concludes the proof.

Unfortunately, in the regime of $p_n$ that we consider, the argument of Czajka and Pandurangan \cite{Czajka2008} breaks down. Their proof crucially relies on concentration of the bin sizes, with confidence $1-o(n^{-2})$ to accommodate a union bound over pairs of distinct vertices $(u,v)$. In our regime, such concentration does not appear to hold. To overcome this difficulty, we create bins of a different sort. Fixing a positive integer $m$, we sort the vertices in $G$ according to their degrees, $\mod m$. Formally, let the $k$th bin be defined as
\[\mathcal{C}_G(k) := \{v \in [n] : \deg(v) = k \mod m\}.\]
We show that the number of vertices in each bin is close to $\frac{n}{m}$; in fact, the bins can be thought of as imitating a uniformly random coloring of the vertices of $G$. For this reason, we say that a vertex $v \in \mathcal{C}_G(k)$ has \emph{color} $k$ relative to $G$. Define $\ell_G(v, A)$ to be a color count list, which records the number of neighbors of $v$ among $A$ in each color: set  $\ell_G(v, A)_k := |\cN(v) \cap \mathcal{C}_G(k) \cap A|$. When $A = [n]$, simply write $\ell_G(v)$.

The signatures of the vertices of $G$ are determined as follows.
\begin{enumerate}
    \item Identify the color classes $\{\mathcal{C}_G(k)\}_{k=0}^{m-1}$.
    \item For each vertex $v$, compute $\ell_G(v)$.
    \item The signature of a vertex $v$ is given by $\left\{\ell_G(u) \right\}_{u \sim v}$.
\end{enumerate}
It is easy to see that the signature of a vertex $v$ is determined only by its depth-$3$ neighborhood (see Figure~\ref{fig:sig_example}). We set $m := \left \lceil \frac{3e}{\epsilon^{\star}(1+\delta)} \right \rceil$, where $\epsilon^{\star}(\cdot)$ is the value from Lemma~\ref{lemma:extreme-degrees}. The following result directly implies Theorem~\ref{thm:3nbrhood_algorithm}. 
\begin{theorem}\label{theorem:unique-signatures}
Let $\frac{(1+\delta) \log n}{n} \leq p_n = o\left(n^{-5/6}\right)$, where $\delta > 0$. Let $G \sim G(n,p_n)$. Then, with high probability, the signatures of all $v \in [n]$ are unique.
\end{theorem}
\noindent 
Note that the existing results for depth-$2$ labeling \cite{Czajka2008,Mossel2019} cover the remaining regime where $p_n = \omega\left( \log^2(n) / n \right)$ and $p_n \leq \frac{1}{2}$.
\begin{figure}[t]
    \centering
    \includegraphics[width=0.7\textwidth]{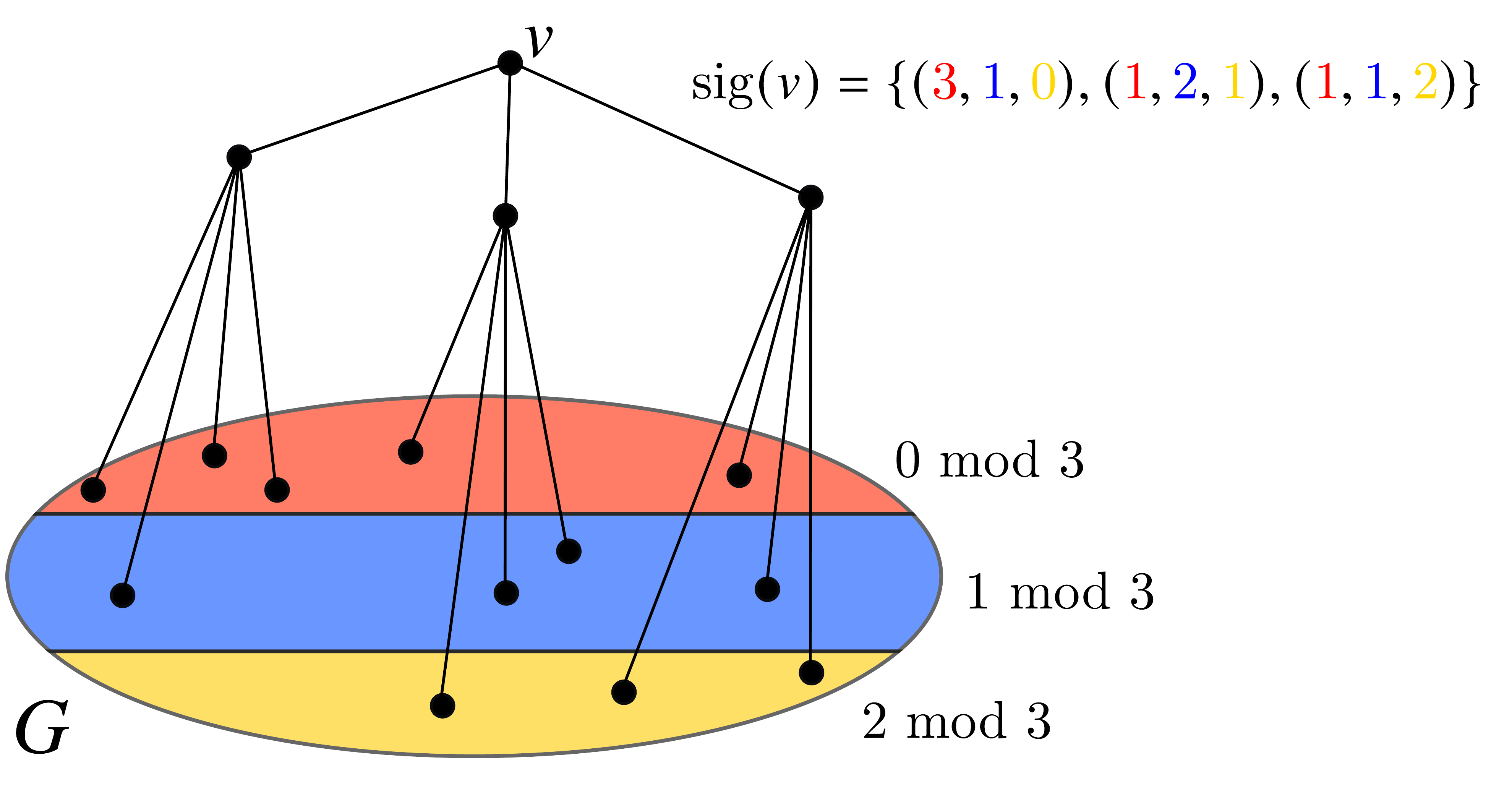}
    \captionsetup{width=.9\linewidth}
    \caption{Example of a vertex signature using 3 colors. In the diagram, the oval region represents the set of vertices in $G$; the red (top) region represents the subset of vertices with degree that is 0 mod 3; the blue (middle) region represents the subset of vertices with degree that is 1 mod 3; the yellow (bottom) region represents the subset of vertices with degree that is 2 mod 3. For each neighbor $u$ of $v$, we construct a 3-element list (the color count list $\ell_G(u)$) which encodes the number of vertices of each color in the neighborhood of $u$, excluding $v$. The signature of $v$ is the collection of tuples corresponding to each of its neighbors.} 
    \label{fig:sig_example}
\end{figure}

In order to give intuition for the proof strategy, suppose that the colors were truly assigned uniformly at random. Fix a pair of vertices $u$ and $v$, we wish to upper-bound the probability that they have the same signatures. Since having the same signature requires having the same degree, let us assume that both $u$ and $v$ have degree $N$. Let $u_1, \dots, u_N$ and $v_1, \dots, v_N$ denote the neighbors of $u$ and $v$, respectively. For simplicity of exposition, we suppose that $u$ and $v$ have no common neighbors, are not themselves connected by an edge, and none of their neighbors are connected. Under these assumptions, if $u$ and $v$ have the same signatures, then there must exist a permutation $\pi: [N] \to [N]$ such that for all $i \in [N]$ and colors $k \in\{0, \dots, m-1\}$, $u_i$ and $v_{\pi(i)}$ have the same number of neighbors of color $k$. If indeed the assignment of colors were independent, then the probability that $u$ and $v$ have the same signature (under $\pi$) is equal to
\begin{align}
\prod_{i=1}^N \prod_{k = 0}^{m-1} \mathbb{P}\left(u_i \text{ and } v_{\pi(i)} \text{ have the same number of neighbors of color } k \right). \label{eq:random-colors}
\end{align}
Let $c_k$ be the number of vertices of color $k$ outside of the neighborhoods of $u$ and $v$, for $k \in \{0, \ldots, m - 1 \}$. Then the number of neighbors of color $k$ is distributed as a binomial random variable with parameters~$(c_k, p_n)$. Due to independence between color counts of neighbors of $u_i$ and $v_{\pi(i)}$, we invoke Corollary~\ref{corollary:binomial-equality} to bound~\eqref{eq:random-colors} from above by 
\begin{align*}
\left(\prod_{k = 0}^{m-1} \frac{e^4}{\sqrt{c_k p_n}}\right)^N.
\end{align*}
Taking a union bound over $N!$ matchings, we can show that $u$ and $v$ have distinct signatures with probability $1 - o(n^{-2})$.

However, since the colors are not uniform, we must adapt the above proof strategy. As before, we begin by revealing the neighborhoods of $u$ and $v$, and assume these vertices have the same number of neighbors. There must be a matching $\pi$ of the neighbors of $u$ to the neighbors of $v$ such that $\ell_G(u_i) = \ell_G(v_{\pi(i)})$ for all $i \in [N]$. Given a matching $\pi$, our goal is then to bound the probability that $\ell_G(u_i) = \ell_G(v_{\pi(i)})$ simultaneously for all pairs of neighbors.
We first show that each vertex $w \in \cN(u) \cup \cN(v)$ has at most a constant number of neighbors within $S= S(u,v) := \{u\} \cup \{v\} \cup \cN(u) \cup \cN(v)$. Therefore, equality of $\ell_G(u_i)$ and $\ell_G(v_{\pi(i)})$ requires near-equality of $\ell_G(u_i, S^c)$ and $\ell_G(v_{\pi(i)}, S^c)$.

But since our bins are defined globally, with respect to the graph $G$, the counts $\{\ell(u_i, S^c)\}_i \cup \{\ell(v_{i}, S^c)\}_i$ are dependent. In order to create independence, we construct ``local'' color classes. Let $H = H(u,v) = G\{S^c\}$ be the graph induced on $S^c$. Note that conditioned on $|S|$, the graph $H$ is an \ER random graph. We can show that with high probability, the local color classes 
\[
\mathcal{C}_{H}(k) := \{w \in S^c : \deg_{H}(w) = k \mod m\} 
\]
are balanced. 

Since the signatures are defined with respect to the global color classes $\{\mathcal{C}_G(k)\}_k$, we need to relate global color count lists to their local counterparts. A crucial observation is the following: if there are exactly $j$ vertices in $\cN(u) \cup \cN(v)$ which are connected to a vertex $z \in \mathcal{C}_H(k)$, then $z$ has color $(k+j) \mod m$ with respect to $G$. In particular, if there is exactly one vertex $w \in \cN(u) \cup \cN(v)$ connected to $z \in \mathcal{C}_H(k)$, then $z$ has color $(k+1) \mod m$ with respect to $G$. We say that a vertex $w \in \cN(u) \cup \cN(v)$ is \emph{colliding} if there is some $w' \in \cN(u) \cup \cN(v)$ where $w' \neq w$ such that $w$ and $w'$ have a common neighbor in $S^c$. We will show that at most a constant number of vertices $w \in \mathcal{N}(u) \cup \cN(v)$ are colliding. If $w$ is a non-colliding vertex, then 
\[\cN(w) \cap C_H(k) = \cN(w) \cap C_G(k+1) \cap S^c,\]
where $k+1$ is taken $\mod m$.
In other words, $\ell_G(w, S^c)$ is a shifted version of $\ell_H(w, S^c)$. See Figure~\ref{fig:sig_shifting} for an illustration of this relation. 

\begin{figure}[t]
    \centering
    \includegraphics[width=0.8\textwidth]{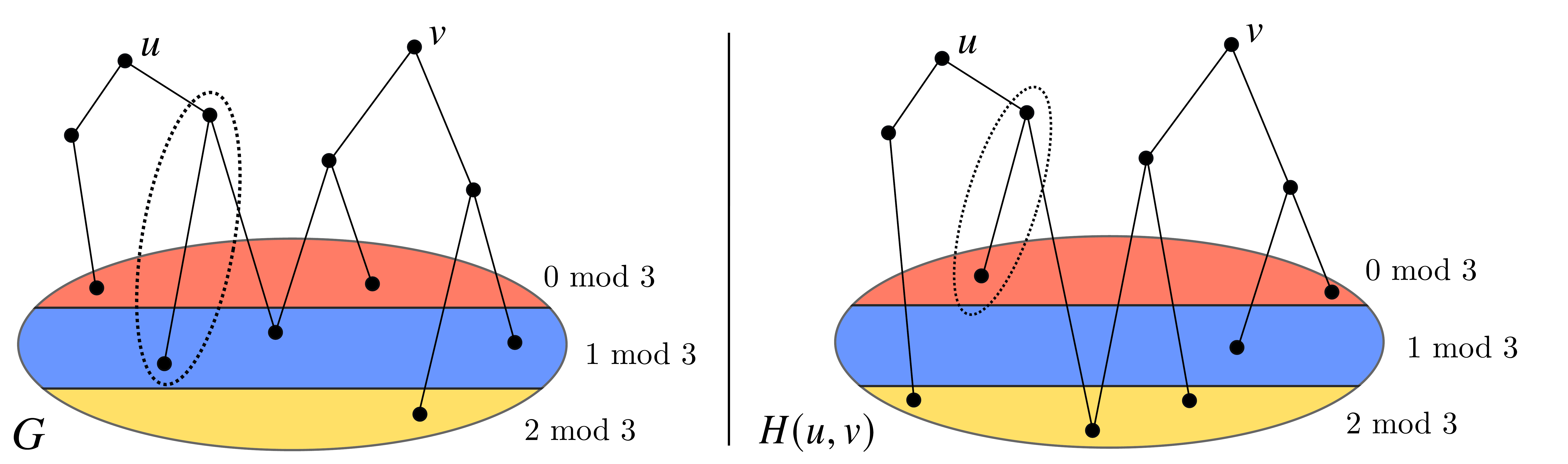}
    \captionsetup{width=.9\linewidth}
    \caption{Illustration of the relation between $\ell_G(w, S^c)$ and $\ell_H(w,S^c)$. For given vertices $u,v$ in $G$, the diagram on the left depicts the color count list for neighbors of $u$ and $v$; the diagram on the right shows the same for connections to the color set with respect to $H(u,v)$ (the version of $G$ with $u, v, \cN(u), \cN(v)$ removed). In particular, the edge circled within the dotted oval in each diagram represents the \emph{same} edge, though the 
    color class of the edge endpoint outside of $u,v, \cN(u)$ and $\cN(v)$ changes between $G$ and $H(u,v)$.} 
    \label{fig:sig_shifting}
\end{figure}

Finally,
the counts $\{\ell_H(w, S^c)\}_{w \in S}$ can be described by independent binomial random variables when we condition on the two graphs induced by $S$ and $S^c$. The value of $m$ is chosen so that the required (near) equalities are sufficiently unlikely. 

\subsection{Isomorphic $2$-neighborhoods}\label{sec:outline-2-neighborhoods}
We next outline the proof of Theorem~\ref{thm:2nbrhood_impossibility}, which states that, when $n p_{n} = o ( \log^{2}(n) / \left( \log \log n \right)^{3} )$, 
there exist distinct yet isomorphic $2$-neighborhoods. 
In fact, the proof will show that the conclusion also holds when $n p_{n} = c \log^{2}(n) / \left( \log \log n \right)^{3}$ 
for constant $c < 1/64$. 
When $n p_{n} \leq \frac{1}{2} \log n$, with high probability there are multiple isolated vertices, so the claim immediately follows; 
consequently we may, 
and thus will, 
assume in the following that 
$n p_{n} \geq \frac{1}{2} \log n$. The proof can be summarized by the following steps.

\begin{enumerate}
    \item 
    \textbf{(Defining ``good'' vertices)} 
    First, 
    define the set of \emph{atypical} vertices, 
    in terms of degrees:
    $$
    A := \{ v : | \deg(v) - (n-1)p| \geq 3\sqrt{np \log \log n}  \}.
    $$
    Now define the set of \emph{``good''} vertices $B$ such that $v \in B$ if and only if $v \notin A$, $v$ has no neighbors in~$A$, and the 2-neighborhood of $v$ is a tree. 
    We will show that $|B| \geq n/2$ with high probability (Lemma~\ref{lemma:B_size} below). 
    Consequently, we may restrict to searching for isomorphic $2$-neighborhoods among good vertices. 
    
    \item 
    \textbf{(Reducing to degree profiles)} 
    For a node $v$ define its \emph{degree profile} to be the list of degrees of the neighbors of $v$, in decreasing order. 
    Since the $2$-neighborhoods of good vertices are trees by definition, 
    it follows that for two nodes in $B$, 
    their $2$-neighborhoods are isomorphic if and only if their degree profiles are identical. 
    Thus it suffices to find distinct good vertices with the same degree profile. 

    \item 
    \textbf{(Bounding the number of degree profiles)} 
    We will show that, 
    when $n p_{n} = \omega ( \log \log n )$, 
    the number of possible degree profiles for a good vertex is at most 
    $\exp ( O ( \sqrt{ n p_{n} \log \log n} \log(n p_{n})  ) )$ 
    (see Lemma~\ref{lemma:number_degree_profiles} below). 
    This is a purely deterministic result that crucially uses the definition of ``goodness''.
    
    \item 
    \textbf{(Pigeonhole principle)} 
    When $n p_{n} = o ( \log^{2}(n) / \left( \log \log n \right)^{3} )$, 
    the previous step implies that the number of possible degree profiles for a good vertex is $o(n)$. Hence by the pigeonhole principle, with high probability there must exist distinct $u,v \in B$ with isomorphic $2$-neighborhoods. (In fact, this gives polynomial lower bounds on the number of vertices with isomorphic $2$-neighborhoods, see Lemma~\ref{lemma:pigeonhole}.)
\end{enumerate}

\subsection{Smoothed analysis of graph isomorphism}\label{sec:smoothed-analysis-outline}
We will focus on the proof of Theorem~\ref{theorem:smoothed-analysis}, Part 1, where $G$ is a union of the deterministic graph $G_1 \in \mathcal{G}_n(\lambda)$ (recall \eqref{eq:deterministic-class}) and the random graph $G_2$. Later, we will show how to modify the argument when $G'$ is the XOR of $G_1$ and $G_2$.

We again show that the signatures of all vertices of $G$ are unique. As before, we show that for a fixed pair $u \neq v$, we have $\mathbb{P}(\text{sig}(u) = \text{sig}(v)) = o(n^{-2})$ and apply the union bound. However, we cannot take a further union bound over matchings of $\mathcal{N}(u)$ to $\mathcal{N}(v)$ as was done in the proof of Theorem \ref{thm:3nbrhood_algorithm}, since now it is possible that $|\mathcal{N}(u)|, |\mathcal{N}(v)| \gg np_n$. Instead, we allow ourselves more randomness by restricting $p_n$ to be $\omega\left(\frac{\log^2 n}{n}\right)$ (and also mildly restrict the upper limit for $p_n$ to bound the random contribution to vertex degrees). The additional randomness allows us to create more color classes, and still ensure that each one is sizeable. Specifically, we show that with $\log n$ color classes, with high probability all of them will be of size $n / \log(n)$ up to first-order terms (see Lemma \ref{lemma:balanced-colors-semirandom}).

Now proving the uniqueness of signatures is simple: we reveal the neighbors of $u$ and $v$, and choose an arbitrary $w \in \mathcal{N}(u) \setminus \mathcal{N}(v)$. We then reveal the rest of the graph, except for some of the random neighbors of $w$. The conditions on $G_1$ and $G_2$ ensure that randomness is ``preserved'' in such a way that equality of signatures implies an unlikely property among the preserved random edges. 
More specifically, fixing vertices $u \neq v$, we will reveal the majority of the edges in the graph $G$ (in a way that depends on $u$ and $v$). First, reveal the graph induced by $u$, $v$, $\mathcal{N}(u)$, and $\mathcal{N}(v)$. Let $w$ be the smallest-index vertex among the set $\mathcal{N}(u) \setminus \left(\mathcal{N}(v) \cup \{v\} \right)$ (we show that there exists such a $w$ in Lemma \ref{lemma:neighborhood_difference}).
Next, reveal the remainder of the graph, but do not reveal edges $(w,w')$ which satisfy all of the following properties:
\begin{enumerate}
    \item \label{item:R_neighborhood}
    $w' \not \in \{u\} \cup \{v\} \cup \mathcal{N}(u) \cup \mathcal{N}(v)$
    \item \label{item:R_2_neighborhood}
    $w'$ is not a neighbor of any $z \in \mathcal{N}(u) \cup \mathcal{N}(v) \setminus \{w\}$
    \item \label{item:R_G1_edges}
    $(w,w')$ is not an edge in $G_1$.
\end{enumerate}
See Figure \ref{fig:random-edges} for an illustration of the revealed information. Let $R = R(u,v)$ denote the set of $w' \in [n]$ such that $(w,w')$ satisfies the above properties. Conditioned on the revealed information, denoted by $\mathcal{I}$, each edge among $\{(w,w') : w' \in R\}$ is present with probability $p_n$, independently.

After revealing most of the graph in the form of $\mathcal{I}$, we can imagine forming the associated signatures of $u$ and $v$. That is, $u$ and $v$ are each associated with a (multi-)set of $m$-length lists, one for each neighbor. How do colors change after revealing the edges $\{(w,w') : w' \in R\}$? Certainly, the color of $w$ may change. Also, the colors of the newly revealed neighbors of $w$ will all be shifted up by $1$. No other colors will change. How do these color changes affect the color count lists of vertices $x \in \mathcal{N}(u) \cup \mathcal{N}(v) \setminus \{w\}$? Any such $x$ that was already known to be connected to $w$ will have its list updated by shifting one unit from one color to another, due to the fact that $w$ may change color. By construction, there will not be any $w' \in R$ such that $w'$ is connected to both $w$ and $x$. Therefore, there will not be any further updates to the lists of $x \in \mathcal{N}(u) \cup \mathcal{N}(v) \setminus \{w\}$; any such $x$ will have its list updated by at most a unit transposition.

Having accounted for the possible changes to color count lists, it follows that if $\{\text{sig}(u)  = \text{sig}(v)\}$, then there must exist $x \sim v$ such that the list associated to $w$ (after revealing everything) must be the same as the list associated to $x$ (after revealing only $\mathcal{I}$), up to a unit transposition.

\begin{figure}[t]
    \centering
    \includegraphics[width=0.8\linewidth]{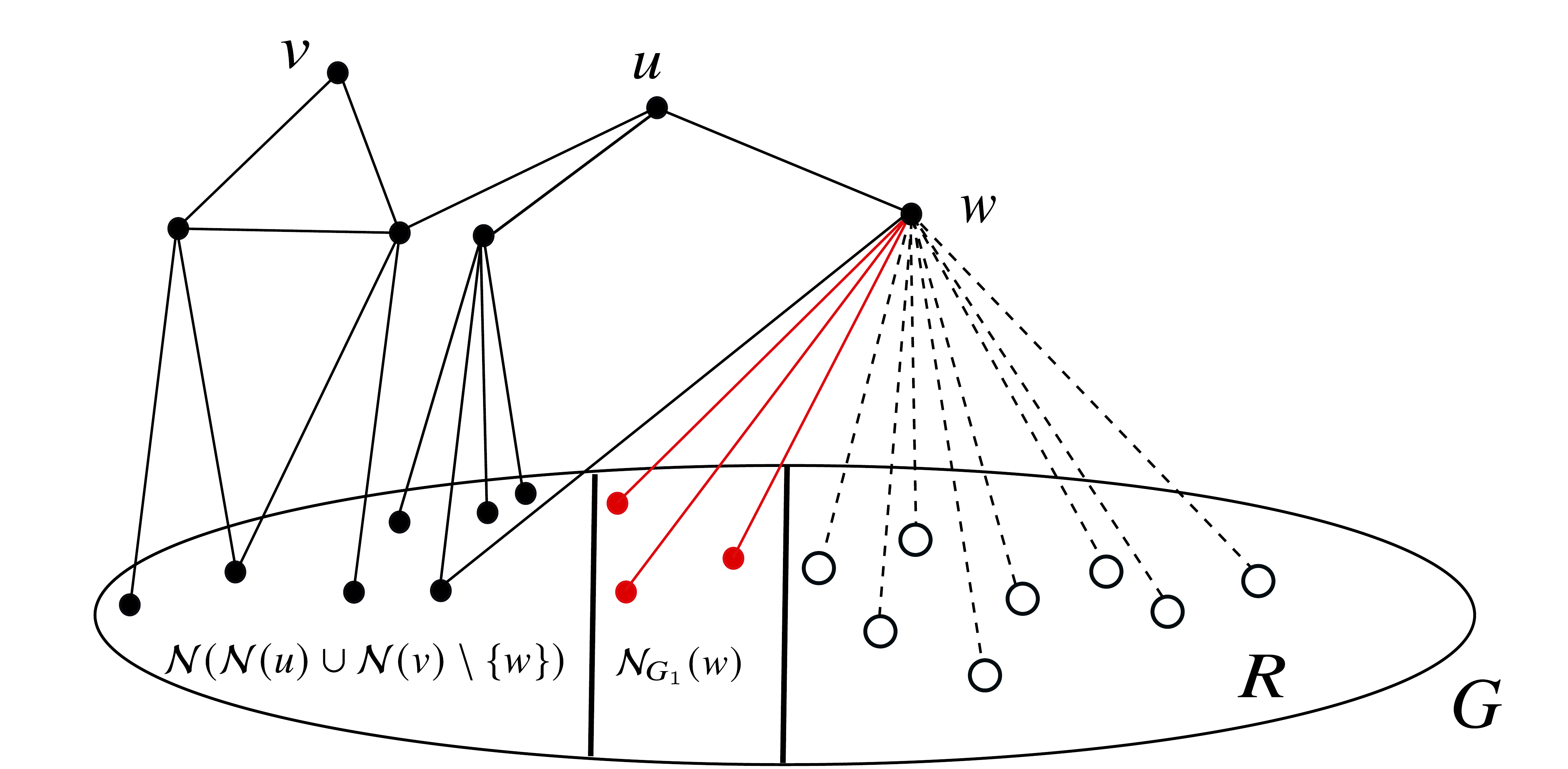}
    \caption{Illustration of the information revealed when proving that $\text{sig}(u) \neq \text{sig}(v)$.
    The oval region represents the set of vertices in $G$ outside of $u$, $v$ and their respective neighborhoods, which encompasses the vertices satisfying Condition~\ref{item:R_neighborhood}.
    The region on the far left is the set of vertices that do not satisfy Condition \ref{item:R_2_neighborhood}. 
    The region in the middle captures neighbors of $w$ in $G_1$ (here, the edges from $G_1$ as well as $w$'s neighborhors in $G_1$ are shown in red), which are the vertices that do not satisfy Condition \ref{item:R_G1_edges}.
    The remaining vertices, which are those that satisfy Conditions \ref{item:R_neighborhood}, \ref{item:R_2_neighborhood} and \ref{item:R_G1_edges}, are shown by the region $R$.
    The formation of potential edges between $w$ and vertices in $R$ (represented by dotted lines) is independent, conditioned on the revealed information.}
    \label{fig:random-edges}
\end{figure}

Let us now sketch the computation of the probability that the color list of $w$ (denoted by $\ell_G(w)$) is the same as the color list of some neighbor of $v$ in $G$ (denoted by $\ell_G(x)$ for $x \sim v$).
Let $H$ be the graph formed from the edges that are revealed by $\mathcal{I}$ (that is, it contains all edges in $G$ except those in the set $\{ (w,w') \in E(G) : w' \in R \}$ ).
Conditioned on $\mathcal{I}$, the exact value of $\ell_G(x)$ is determined by the number of neighbors $w$ has in each color class in $R$ with respect to $H$.
We show that $R = n - o(n)$ (see Lemma \ref{lemma:size-R}) and that the color classes are balanced in $R$; that is, $| \mathcal{C}_H(m,k) \cap R | \sim n / m$ for all $k \in \{0, \ldots, m - 1 \}$.
As a result, the number of neighbors $w$ has of color $k$ in $H$ is (approximately) given by $X_k \sim \mathrm{Bin}(n/m, p_n)$, where the $X_k$'s are independent.
The probability that $\ell_G(w)$ takes on a particular value reduces to the probability that the $X_k$'s take on a particular (deterministic) value. 
By Lemma \ref{lemma:binomial-upper-bound}, this probability is at most 
\[
\left( \frac{e^4}{\sqrt{np_n / m}} \right)^m.
\]
Since $\ell_G(x)$ can be constructed by taking at most a unit transposition of $\ell_H(x)$ (which is measurable with respect to $\mathcal{I}$), there are at most $m^2$ possibilities for the value of $\ell_G(x)$. 
Taking a union bound over these possibilities shows that
\[
\mathbb{P}( \ell_G(w) = \ell_G(x) \vert \mathcal{I} ) \le m^2  \left( \frac{e^4}{\sqrt{np_n / m}} \right)^m.
\]
To relate this probability to the signatures of $u$ and $v$, we bound the probability that $\ell_G(w) = \ell_G(x)$ over \emph{all} neighbors $x$ of $v$.
It follows that
\[
\mathbb{P} ( \text{sig}(u) = \text{sig}(v) \vert \mathcal{I} ) \le \mathbb{P} ( \exists x \in \mathcal{N}(v) : \ell_G(w) = \ell_G(x) \vert \mathcal{I} ) \le \sum_{x \in \mathcal{N}(v) } \mathbb{P}( \ell_G(w) = \ell_G(x) \vert \mathcal{I} ) \le n m^2  \left( \frac{e^4}{\sqrt{np_n / m}} \right)^m.
\]
In the last inequality, we have used that $| \mathcal{N}(v) | \le n$.
Finally, due to the choices $m = \log n$ and $np_n = \omega (\log^2 n)$, the probability that the signatures of $u$ and $v$ are distinct is $1 - o(n^{-2})$.
A union bound over $u,v \in [n]$ proves that all signatures in $G$ are distinct.

\section{Preliminaries}\label{sec:preliminaries}

\subsection{Probability bounds}

We first collect standard Chernoff bounds that we use in our arguments (see, e.g.,~\cite[Theorems~4.4,~4.5]{mitzenmacher2017probability}).

\begin{lemma}[Chernoff bounds]
Let $X_1, \dots, X_n$ be mutually independent Bernoulli random variables, where $X_i \sim \mathrm{Bern}(p_i)$. Let $X := \sum_{i=1}^n X_i$ and let $\mu := \sum_{i=1}^n p_i = \mathbb{E}[X]$. 
\begin{enumerate}
    \item For any $\delta > 0$,
\begin{align}
\mathbb{P}\left(X \geq (1+\delta) \mu \right) &\leq \left(\frac{e^{\delta}}{(1+\delta)^{1+\delta}} \right)^{\mu}. \label{eq:Chernoff-upper}
\end{align}
\item For any $\delta \in (0,1)$,
\begin{align}
\mathbb{P}\left(X \leq (1-\delta) \mu \right) &\leq \left(\frac{e^{-\delta}}{(1-\delta)^{1-\delta}} \right)^{\mu}. \label{eq:Chernoff-lower}
\end{align}
\end{enumerate}
\end{lemma}
\begin{corollary}\label{corollary:Chernoff}
Let $X_1, \dots, X_n$ be mutually independent Bernoulli random variables, where $X_i \sim \mathrm{Bern}(p_i)$. Let $X := \sum_{i=1}^n X_i$ and let $\mu := \sum_{i=1}^n p_i = \mathbb{E}[X]$. If $0 < \mu < t$, then
\[
\mathbb{P}\left(X \geq t \right) \leq \left(\frac{e \mu}{t}\right)^t.
\]
\end{corollary}
\begin{proof}
The claim follows from~\eqref{eq:Chernoff-upper} by plugging in $\delta = (t-\mu)/\mu$. 
\end{proof}

The next lemma gives an upper bound on binomial probabilities. 
Let $b(k;n,p)$ denote the probability that $X \sim \text{Bin}(n,p)$ takes value $k \in \{0, 1, \dots, n\}$.
\begin{lemma}\label{lemma:binomial-upper-bound}
If $np \geq 2$ and $p \leq \frac{1}{3}$, then $\max_{k} b(k; n,p) \leq \frac{e^4}{\sqrt{np}}$.
\end{lemma}
\begin{proof}
The mode of the binomial distribution is $\lfloor np \rfloor$ or $\lceil np \rceil$. We use the following exact bounds on factorials: 
\[
\sqrt{2\pi n} \left(\frac{n}{e} \right)^n < \sqrt{2\pi n} \left(\frac{n}{e} \right)^n e^{\frac{1}{12n + 1}} < n! < \sqrt{2\pi n} \left(\frac{n}{e} \right)^n e^{\frac{1}{12n}}\leq \sqrt{2\pi n} \left(\frac{n}{e} \right)^n e^{\frac{1}{12}}.
\]
Then 
\[
\binom{n}{m} 
= \frac{n!}{m! (n-m)!} 
\leq \frac{\sqrt{2\pi n} \left(\frac{n}{e} \right)^n e^{\frac{1}{12}}}{\sqrt{2\pi m} \left(\frac{m}{e} \right)^m  \cdot \sqrt{2\pi (n-m)} \left(\frac{n-m}{e} \right)^{n-m} } 
= \frac{e^{\frac{1}{12}}}{\sqrt{2\pi}} \sqrt{\frac{n}{m(n-m)}}  n^n m^{-m} (n-m)^{-(n-m)}. 
\]
Therefore, 
\begin{equation}
b(m; n,p) 
= \binom{n}{m} p^{m} (1-p)^{n-m}
\leq \frac{e^{\frac{1}{12}}}{\sqrt{2\pi}} \sqrt{\frac{n}{m(n-m)}} \left(\frac{n(1-p)}{n-m} \right)^n \left(\frac{p(n-m)}{m(1-p)} \right)^m. \label{eq:max-bin1}
\end{equation}
Consider $m \in \{\lfloor np \rfloor, \lceil np \rceil\}$.
Since $m \geq np - 1 \geq \frac{1}{2} np$ and $n-m \geq n - np - 1 \geq \frac{n}{2}$, we have that
\begin{align}
\sqrt{\frac{n}{m(n-m)}} &\leq  \frac{2}{\sqrt{np}}. \label{eq:max-bin2}
\end{align}
Using $\log x \leq x-1$ and also $|m-np| \leq 1$, we have that  
\begin{equation}
\left(\frac{n(1-p)}{n-m} \right)^n 
\leq \exp \left(n \left(\frac{n(1-p)}{n-m} -1 \right) \right) 
= \exp \left(\frac{n (m-np)}{n-m} \right) 
\leq \exp \left(\frac{n}{n-m} \right)
\leq e^2 \label{eq:max-bin3}
\end{equation}
Similarly, 
\begin{equation}
\left(\frac{p(n-m)}{m(1-p)} \right)^m 
\leq \exp \left(m \left(\frac{p (n-m)}{m(1-p)} - 1 \right) \right)
= \exp \left(\frac{np-m}{1-p} \right)
\leq e^{\frac{1}{1-p}} 
\leq e^{3/2}. \label{eq:max-bin4}
\end{equation}
The result follows by substituting \eqref{eq:max-bin2}--\eqref{eq:max-bin4} into \eqref{eq:max-bin1}.
\end{proof}

\begin{corollary}\label{corollary:binomial-equality}
Let $X, Y \sim \mathrm{Bin}(n,p)$ be independent, where $np \geq 2$ and $p \leq \frac{1}{3}$. Then for any $a \geq 0$,
\[\mathbb{P}\left(\left|X - Y \right| \leq a \right) \leq \frac{(2a+1)e^4}{\sqrt{np}}.\]
\end{corollary}
\begin{proof} 
Conditioning on $Y$ we have that 
\[
\mathbb{P}\left(|X - Y| \leq a \right) 
= \sum_{k = 0}^n \mathbb{P}\left(|X - k| \leq a \right) \mathbb{P}(Y = k)
= \sum_{k = 0}^n \mathbb{P}(Y = k) 
\sum_{\ell = k - a}^{k+a}\mathbb{P}\left(X = \ell \right). 
\]
By Lemma~\ref{lemma:binomial-upper-bound} we have that 
$\mathbb{P} \left( X = \ell \right) \leq e^{4} / \sqrt{np}$; 
plugging this into the display above, the claim follows.  
\end{proof}

\subsection{Properties of Erd\H{o}s-R\'{e}nyi random graphs}

The following lemma states that, with high probability, the degree of every vertex is on the order of $n p_{n}$. 

\begin{lemma}[Extreme degrees]\label{lemma:extreme-degrees}
Fix $a > 1$, 
let $p_n \geq \frac{a \log n}{n}$, 
and let $G \sim G\left(n, p_n\right)$. 
Then the following two statements hold: 
\begin{enumerate}
    \item With probability at most $n^{1-a}$, $\max_{i \in [n]} \deg(i) \leq en p_n$. 
    \item There exists $\epsilon^{\star} = \epsilon^{\star}(a) > 0$ such that, with high probability, $\min_{i \in [n]} \deg(i) \geq \epsilon^{\star} n p_n$.
\end{enumerate}
\end{lemma}
\begin{proof}
To prove the first part, fix $i \in [n]$. Applying the Chernoff upper bound~\eqref{eq:Chernoff-upper} with $\delta = e - 1$, we obtain  
\begin{equation*}
\mathbb{P}\left(\deg(i) \geq e n p_n \right) 
\leq \left(\frac{e^{e-1}}{e^e} \right)^{n p_n}
= e^{-n p_n}
\leq e^{-a \log n}
= n^{-a}.
\end{equation*}
Since $a > 1$, the first claim follows by a union bound.   

To prove the second part, 
we apply the Chernoff lower bound \eqref{eq:Chernoff-lower}. For any $\delta \in (0,1)$ and $i \in [n]$, 
\begin{equation*}
\mathbb{P}\left(\deg(i) \leq (1-\delta) n p_n \right) \leq \left(\frac{e^{-\delta}}{(1-\delta)^{1-\delta}} \right)^{n p_n}
\leq \left(\frac{e^{-\delta}}{(1-\delta)^{1-\delta}} \right)^{a \log n}
= n^{-a(\delta + (1-\delta) \log(1-\delta))}.
\end{equation*}
Since 
$\delta + (1-\delta) \log(1-\delta) \to 1$
as $\delta \to 1$, 
there exists 
$\delta^{\star} = \delta^{\star}(a) \in (0,1)$ such that 
$-a(\delta^{\star} + (1-\delta^{\star}) \log(1-\delta^{\star})) < -1$. 
Let $\epsilon^{\star} = 1 - \delta^{\star}$. The proof is complete by a union bound.
\end{proof}

Our next result establishes a tail bound for the number of common neighbors of two given vertices. In the following, recall that $\cN(u)$ is the set of neighbors of $u$ in $G$.

\begin{lemma}\label{lemma:common-neighbors}
Let $p_n = o(n^{-\frac{5}{6}})$ and $G \sim G(n,p_n)$. Then, with high probability, 
$\max_{u \neq v \in [n]} | \cN(u) \cap \cN(v)| \leq 2$.
\end{lemma}
\begin{proof}
Fix distinct $u, v \in [n]$. Then $| \cN(u) \cap \cN(v)| \sim \text{Bin}(n-2, p_n^2)$. 
Let $\mu := \mathbb{E}[|\cN(u) \cap \cN(v) |] = (n-2) p_n^2$ 
and note that $\mu = o(n^{-2/3})$ 
since 
$p_n = o(n^{-5/6})$. Using Corollary~\ref{corollary:Chernoff}, 
we obtain that for any $C > 0$,
\[
\mathbb{P}\left(| \cN(u) \cap \cN(v) | \geq C \right) 
\leq \left(\frac{e \mu}{C} \right)^{C}.
\]
Taking $C = 3$, we get that 
\[
\mathbb{P}\left(| \cN(u) \cap \cN(v) | \geq 3 \right) \leq \left( \frac{e}{3} np_n^2 \right)^3 = o\left(n^{-2}\right).\]
Since $| \cN(u) \cap \cN(v) |$ is integer, we conclude that $\mathbb{P}\left(| \cN(u) \cap \cN(v) | > 2 \right) = o\left(n^{-2}\right)$.
The claim follows by a union bound over pairs of distinct vertices.
\end{proof}

\section{Canonical labeling from $3$-neighborhoods: Proof of Theorem \ref{thm:3nbrhood_algorithm}}\label{sec:3nbrhood}

\subsection{Useful events}
We identify several events which will be used throughout the proof of Theorem~\ref{theorem:unique-signatures}. Let \[\mathcal{E}_1 :=  \bigcap_{u \neq v \in [n]} \{|\cN(u) \cap \cN(v)| \leq 2\}.\] Let $\mathcal{E}_2(u) := \{\epsilon^{\star} n p_n \leq \deg(u) \leq e n p_n\}$ and $\mathcal{E}_2 := \bigcap_{u \in [n]} \mathcal{E}_2(u)$. 

Let
\[\mathcal{E}_3(u,v) := \bigcap_k \left\{\left || \mathcal{C}_{H(u,v)}(m,k)| - \frac{n}{m} \right| \leq  \frac{\epsilon n}{m}\right\},\]
where $\epsilon > 0$ satisfies $\frac{9}{\sqrt{1-\epsilon}} < 10$. Finally, let 
$\mathcal{E}_4(u,v)$ be the event that each vertex in $S^c$ has at most four neighbors in $\cN(u) \cup \cN(v)$, and at most five vertices in $S^c$ have at least two neighbors in $\cN(u) \cup \cN(v)$.

For convenience, let $\mathcal{E}(u,v) := \mathcal{E}_1 \cap \mathcal{E}_2 \cap \mathcal{E}_3(u,v) \cap \mathcal{E}_4(u,v)$. 
By Lemmas~\ref{lemma:common-neighbors} and~\ref{lemma:extreme-degrees}, we have that 
$\mathbb{P}(\mathcal{E}_1) = 1 - o(1)$ 
and 
$\mathbb{P}(\mathcal{E}_2) = 1-o(1)$. 
The following lemmas control $\mathcal{E}_3(u,v)$ and $\mathcal{E}_4(u,v)$ 
(on the event~$\cE_{1} \cap \cE_{2}$). 

\begin{lemma}\label{lemma:balanced-colors}
Suppose that $p_n = \omega(1/n)$ and $p_n = o(1 / \log(n))$. Let $G \sim G(n,p_n)$. Fix $m \in \mathbb{N}$. Then for $\epsilon > 0$,
\[\mathbb{P}\left(\bigcap_{k \in \{0,1,\dots, m-1\}} \left\{\left| |\mathcal{C}_G(m, k)| - \mathbb{E}[|\mathcal{C}_G(m, k)|] \right| \leq \frac{\epsilon n}{m} \right\}\right) \geq 1 - o(n^{-2}).\]
Also, for any $u \neq v$,
\begin{equation}
\mathbb{P}\left(\mathcal{E}_3(u,v)^c \cap \mathcal{E}_2 \right) = o\left(n^{-2} \right). \label{eq:balanced-colors}
\end{equation}
\end{lemma}

\begin{lemma}\label{lemma:H-connectivity}
Let $\frac{(1+\delta) \log n}{n} \leq p_n = o\left(n^{-5/6}\right)$, where $\delta > 0$. Then for any $u \neq v$,
\begin{align*}
\mathbb{P}\left(\mathcal{E}_4(u,v)^c \cap \mathcal{E}_1 \cap \mathcal{E}_2 \right) &= o(n^{-2}).
\end{align*}
\end{lemma}
The proofs of Lemmas \ref{lemma:balanced-colors} and \ref{lemma:H-connectivity} are deferred to Section \ref{sec:supporting}.

\subsection{Relating color counts} 
The following result relates color counts relative to $G$ to color counts relative to $H$. We treat $\ell_G(w)$ and $\ell_H(w)$ as vectors in $\mathbb{R}^m$. Comparison operators are taken componentwise.
\begin{lemma}\label{lemma:shifted-colors}
Condition on a realization of $G$ where the event $\mathcal{E}_1 \cap \mathcal{E}_4(u,v) \cap \{\deg(u) = \deg(v) = r\}$ holds. Label the neighbors of $u$ as $u_1, \ldots, u_r$ and the neighbors of $v$ as $v_1, \ldots, v_r$. Let $\pi: [r] \to [r]$ be a matching of the neighbors of $u$ to the neighbors of $v$. Then there are at least $r - 28$ pairs of the form $(u_i, v_{\pi(i)})$ such that:
\begin{enumerate}
    \item $u_i \neq v$ and $v_{\pi(i)} \neq u$.
    \item Neither $u_i$ nor $v_{\pi(i)}$ is a common neighbor of $u$ and $v$.
    \item $\ell_G(u_i, S^c) \leq \ell_G(u_i)\leq \ell_G(u_i, S^c) + 4\cdot \mathbf{1}_m$
    \item $\ell_G(v_{\pi(i)}, S^c) \leq \ell_G(v_{\pi(i)}) \leq \ell_G(v_{\pi(i)}, S^c) + 4 \cdot \mathbf{1}_m$
    \item $\left[\ell_H(u_i,S^c)\right]_k = \left[\ell_G(u_i,S^c)\right]_{k-1}$
    \item $\left[\ell_H(v_{\pi(i)},S^c)\right]_k = \left[\ell_G(v_{\pi(i)},S^c)\right]_{k-1}$
\end{enumerate}
In particular, the final four statements imply that
\begin{equation}
\{\ell_G(u_i) = \ell_G(v_{\pi(i)})\} \implies \{\left|\ell_H(u_i, S^c) - \ell_H(v_{\pi(i)}, S^c)\right| \leq 4 \cdot \mathbf{1}_m\}. \label{eq:x-y}
\end{equation}
\end{lemma}
\begin{proof}
If $u \nsim v$, then the first statement holds for all pairs. Otherwise, there are at most two pairs $(u_i, v_{\pi(i)})$ for which the statement fails to hold.

Under $\mathcal{E}_1$, there are at most two common neighbors of $u$ and $v$. Therefore, there are at most four pairs $(u_i, v_{\pi(i)})$ which do not satisfy the second statement. 

We next show that all pairs satisfying the first two statements also satisfy the third statement. Consider $u_i \sim u$ (where we can assume $u_i \neq v$ and $u_i \nsim v$ since we have already accounted for those pairs). Clearly $\ell_G(u_i, S^c) \leq \ell_G(u_i)$. On the other hand, since the event $\cE_1$ holds, we know that $| \cN(u_i) \cap \cN(u) | \le 2$ and that $| \cN(u_i) \cap \cN(v) | \le 2$; hence all entries of $\ell_G(u_i, S^c)$ increase by at most 4 when we account for connections to $S$ in the color count list. This leads to the desired upper bound $\ell_G(u_i) \le \ell_G(u_i, S^c) + 4 \cdot \mathbf{1}_m$. By similar reasoning, all pairs satisfying the first two statements also satisfy the fourth.

Coming to the final pair of statements, note that  $\left[\ell_H(w,S^c)\right]_k = \left[\ell_G(w,S^c)\right]_{k-1}$ holds for all $w \in \cN(u) \cup \cN(v)$ which are not colliding (see Section \ref{subsec:outline_canonical_labeling} for the definition of a colliding vertex). We therefore upper-bound the number of colliding vertices. Under $\mathcal{E}_4(u,v)$, there are at most 5 vertices in $S^c$ with two or more neighbors in $S$, and each of these vertices is connected to at most 4 vertices in $S$. This implies that at most 20 vertices among $\cN(u) \cup \cN(v)$ are colliding.
Since at most two of these vertices are shared neighbors of $u$ and $v$, there are at most $22$ pairs with a colliding element (since a given shared neighbor $u_i = v_{\pi(j)}$ can appear in two pairs).

It follows that at most $2+4 + 22 = 28$ pairs of vertices do not satisfy the statements.
\end{proof}

\subsection{Proving the uniqueness of signatures}
\begin{proof}[Proof of Theorem \ref{theorem:unique-signatures}]
Fix distinct $u, v \in [n]$, with $u < v$. It suffices to show that the signatures of $u$ and $v$ are unique with probability $1-o(n^{-2})$. Since $\{\text{sig}(u) = \text{sig}(v)\}$ implies $\{\deg(u) = \deg(v)\}$, we have
\begin{align}
\mathbb{P}(\text{sig}(u) = \text{sig}(v) ) &= \mathbb{P}(\{\text{sig}(u) = \text{sig}(v)\} \cap \{\deg(u) = \deg(v)\} ). \label{eq:equal-degrees}
\end{align}

In light of $\mathcal{E}_2$, we condition on $\deg(u) = \deg(v) = r$, where $\epsilon^{\star} np_n \leq r \leq e np_n$. Let $\{u_i\}_{i=1}^r$ and $\{v_i\}_{i=1}^r$ denote the neighbors of $u$ and $v$. In order for the signatures of $u$ and $v$ to be equal, there must be some matching $\pi : [r] \to [r]$ of $\{u_i\}_{i\in [r]}$ to $\{v_i\}_{i\in [r]}$ such that $\ell_G(u_i) = \ell_G(v_{\pi(i)})$ for all $i \in [r]$. We therefore fix $\pi : [r] \to [r]$, and upper-bound the probability of the event $\cap_{i \in [r]}\left\{\ell_G(u_i) = \ell_G(v_{\pi(i)})\right\}$, which we denote by $\{\text{sig}(u) =_{\pi} \text{sig}(v)\}$.

There are several difficulties that arise when computing the above event. First, if $u$ and $v$ are connected by an edge, $v = u_i$ for some $i$ and $u = v_j$ for some $j$. Another difficulty arises from shared neighbors of $u$ and $v$; in that case, we will have $u_i = v_j$ for some $(i,j)$ pairs. The final, main difficulty is that the random lists $\{\ell_G(w)\}_{w \in A}$ are dependent for any set $A$. 

In order to remedy the first two difficulties, we simply filter the pairs $\{u_i, v_{\pi(i)}\}$. Let 
\[S' = S \setminus \left(\{u,v\} \cup (\cN(u) \cap \cN(v)) \right).\] 
That is, $S'$ filters $S$ to remove $u$, $v$, and the common neighbors of $u$ and $v$. Observe that since $\deg(u) = \deg(v)$, we have $|\cN(u) \cap S'| = |\cN(v) \cap S'|$. Recalling that $\mathcal{E}_1$ implies $|\cN(u) \cap \cN(v)| \leq 2$, we condition on $|\cN(u) \cap S'| = |\cN(v) \cap S'| = r'$, where $\epsilon^{\star} n p_n - 3 \leq r' \leq enp_n$ since the event $\cE_2$ holds (here we subtract $3$ to remove common neighbors of $u$ and $v$, and to account for the case where $u$ and $v$ may be neighbors). Without loss of generality, let $\{u_i\}_{i=1}^{r'}$ denote $N(u) \cap S'$. To remedy the final difficulty, we will analyze the counts $\{\ell_H(w)\}_{w \in S'}$ rather than $\{\ell_G(w)\}_{w \in S'}$.  
By construction of $S'$, 
\[\left(\{u_i\}_{i \in [r']}\right) \cap \left(\{v_{\pi(i)}\}_{i \in [r']}\right) = \emptyset\]
and $\{\ell_H(w)\}_{w \in S'}$ is a collection of independent variables, conditioned on $S$.

Given a subset $A \subseteq [r']$, consider the event 
\[
\mathcal{M}_A := \bigcap_{i \in A} \{\left|\ell_H(u_i, S^c) - \ell_H(v_{\pi(i)}, S^c)\right| \leq 4 \cdot \mathbf{1}_m\},
\]
where the absolute difference is taken componentwise. (We write $\ell_H(w, S^c)$ rather than $\ell_H(w)$ to emphasize that the color count list is based on the vertices in $S^c$ only.)
By Lemma \ref{lemma:shifted-colors},
\begin{equation}
\{\text{sig}(u) =_{\pi} \text{sig}(v)\} \cap \{\deg(u) = \deg(v) =r\} \cap \mathcal{E}(u,v) \implies \cE(u,v) \cap \bigcup_{A \subset [r'] : |A| = r' - 28 }\mathcal{M}_A. \label{eq:implication}  
\end{equation}
We remark that since $r' = \Omega(np_n)$ on the event $\cE_2$, $r'$ is larger than $28$ for $n$ sufficiently large on this event. Fix $A \subseteq [r']$ where $|A| = r' - 28$. Our goal is to upper-bound $\mathbb{P}(\mathcal{M}_A)$. To this end, choose $\epsilon > 0$ satisfying $\frac{9}{\sqrt{1-\epsilon}} < 10$. We condition on
\begin{equation}
\mathcal{B} = \mathcal{B}(c_0, \dots, c_{m-1}) := \bigcap_{k=0}^{m-1}\{|\mathcal{C}_H(k)| = c_k\}, \label{eq:color-conditioning}
\end{equation}
where $\left| c_k - \frac{n}{m}\right| \leq \frac{\epsilon n}{m}$ for all $k \in \{0, 1, \dots, m-1\}$, as guaranteed by $\mathcal{E}_3(u,v) \supseteq \mathcal{E}(u,v)$. In order to mimic the analysis for random colors (Equation \eqref{eq:random-colors}), we introduce independent binomial random variables $\{U_{ik}: i \in [r'], k \in \{0,1,\dots, m-1\} \}$ and $\{V_{ik}: i \in [r'], k \in \{0,1,\dots, m-1\} \}$, where
\begin{align*}
U_{ik}, V_{ik} \sim \text{Bin}(c_k, p_n).
\end{align*}
Then conditioned on \eqref{eq:color-conditioning}, provided $|c_k - n/m| \le \epsilon n / m$ for each $k = 0, \ldots, m-1$, 
\begin{align*}
\mathbb{P}(\cE(u,v) \cap \mathcal{M}_A \mid \mathcal{B})  &= \prod_{i=1}^{r' -28} \prod_{k=0}^{m-1} \mathbb{P}\left(|U_{ik} - V_{ik}| \leq 4 \right)  \leq \left(\frac{9e^4}{\sqrt{(1-\epsilon) \frac{np_n}{m}}} \right)^{(r' - 28)m},
\end{align*}
The Law of Total Probability then implies
\begin{equation}
\label{eq:MA_probability}
\mathbb{P}( \cE(u,v) \cap \cM_A) 
\le \left(\frac{9e^4}{\sqrt{(1-\epsilon) \frac{np_n}{m}}} \right)^{(r' - 28)m}.
\end{equation}

Invoking \eqref{eq:implication} along with a union bound over sets $A$, we obtain
\begin{align*}
\mathbb{P}\left(\{\text{sig}(u) =_{\pi} \text{sig}(v)\} \cap \{\deg(u) = \deg(v) =r \} \cap \mathcal{E}(u,v)\right) & \le \mathbb{P} \left( \cE(u,v) \cap \bigcup_{A \subset [r'] : |A| = r' - 28} \cM_A \right) \\
& \le \sum_{A \subset [r'] : |A| = r' - 28} \mathbb{P} ( \cE(u,v) \cap \cM_A ) \\
& \leq \binom{r'}{r'-28}  \left(\frac{9e^4}{\sqrt{(1-\epsilon) \frac{np_n}{m}}} \right)^{(r' - 28)m}.
\end{align*}
Above, the inequality on the second line is due to a union bound, and the third inequality uses the bound in~\eqref{eq:MA_probability}.

Next, notice that $\text{sig}(u) = \text{sig}(v)$ if there exists $\pi : [r] \to [r]$ such that $\text{sig}(u) =_\pi \text{sig}(v)$. Taking a union bound over such $\pi$, we obtain
\begin{align*}
\mathbb{P}\left(\{\text{sig}(u) = \text{sig}(v)\} \cap \{\deg(u) = \deg(v) = r\} \cap \mathcal{E}(u,v)\right) &\leq r! \binom{r'}{r'-28} \left(\frac{9e^4}{\sqrt{(1-\epsilon) \frac{np_n}{m}}} \right)^{(r' - 28)m}\\
&\leq r^r \cdot r'^{28} \left(\frac{9e^4}{\sqrt{(1-\epsilon) \frac{np_n}{m}}} \right)^{(r' - 28)m}.
\end{align*}
Since $\{\deg(u) = \deg(v) = r\} \cap \mathcal{E}(u,v)$ implies $\deg(u) = \deg(v) \leq enp_n$ and $|\cN(u) \cap \cN(v)| \leq 2$, we need only consider $r, r'$ satisfying $r \leq enp_n$ and $\epsilon^{\star} np_n - 3 \leq r' \leq e np_n$. Applying these bounds, we see that 
\begin{align*}
\mathbb{P}\left(\{\text{sig}(u) = \text{sig}(v)\} \cap \{\deg(u) = \deg(v) = r\} \cap \mathcal{E}(u,v) \right) 
&\leq (enp_n)^{enp_n + 28} \left(\frac{9e^4}{\sqrt{(1-\epsilon) \frac{np_n}{m}}} \right)^{(\epsilon^{\star} np_n - 31) m} \nonumber\\
&\leq (enp_n)^{enp_n + 28} \left(\frac{10e^4}{\sqrt{ \frac{np_n}{m}}} \right)^{(\epsilon^{\star} np_n - 31) m} \nonumber\\
&\leq (enp_n)^{enp_n + 28} \left(\frac{1000e^{12}}{\left( \frac{np_n}{m}\right)^{3/2}} \right)^{enp_n - 31 e/\epsilon^{\star}}.
\end{align*}
Accounting for the possible values of $r$ implied by $\mathcal{E}(u,v)$, we conclude that
\begin{align}
\mathbb{P}\left(\{\text{sig}(u) = \text{sig}(v)\} \cap \{\deg(u) = \deg(v)\} \cap \mathcal{E}(u,v) \right) &\leq (enp_n)^{enp_n + 29} \left(\frac{1000e^{12}}{\left( \frac{np_n}{m}\right)^{3/2}} \right)^{enp_n - 31 e/\epsilon^{\star}} \nonumber\\
&= (np_n)^{-\Theta(np_n)} \nonumber \\
&= o(n^{-2}).\label{eq:sig-bound}
\end{align}

Finally, we bound the probability of interest: 
\begin{align}
&\mathbb{P}\left(\bigcup_{u \neq v} \{\text{sig}(u) = \text{sig}(v)\} \right) =\mathbb{P}\left(\bigcup_{u \neq v} \{\text{sig}(u) = \text{sig}(v)\} \cap  \{\deg(u) = \deg(v)\} \right) \label{eq:final-0}\\
&\leq \mathbb{P}\left( \mathcal{E}_1^c \cup \mathcal{E}_2^c \right) + \mathbb{P}\left(\left(\bigcup_{u \neq v} \{\text{sig}(u) = \text{sig}(v)\}\cap  \{\deg(u) = \deg(v)\} \right) \cap \mathcal{E}_1 \cap \mathcal{E}_2 \right) \label{eq:final-1}\\
&\leq o(1) + \sum_{u \neq v} \mathbb{P}\left(\{\text{sig}(u) = \text{sig}(v)\} \cap  \{\deg(u) = \deg(v)\}   \cap \mathcal{E}_1 \cap \mathcal{E}_2 \right)  \label{eq:final-2}\\
&\leq o(1) + \sum_{u \neq v} \Bigg[ \mathbb{P}\left(\{\text{sig}(u) = \text{sig}(v)\} \cap  \{\deg(u) = \deg(v)\}   \cap \mathcal{E}(u,v)\right) \nonumber \\
&\qquad \qquad \qquad+ \mathbb{P}\left(\{\text{sig}(u) = \text{sig}(v)\} \cap  \{\deg(u) = \deg(v)\}   \cap (\mathcal{E}_3(u,v)^c \cap \mathcal{E}_1 \cap \mathcal{E}_2) \right) \nonumber \\
&\qquad \qquad \qquad+ \mathbb{P}\left(\{\text{sig}(u) = \text{sig}(v)\} \cap  \{\deg(u) = \deg(v)\}   \cap (\mathcal{E}_4(u,v)^c \cap \mathcal{E}_1 \cap \mathcal{E}_2) \right)\Bigg] \label{eq:final-3}\\
&\leq o(1) + \sum_{u \neq v} \Bigg[ \mathbb{P}\left(\{\text{sig}(u) = \text{sig}(v)\} \cap  \{\deg(u) = \deg(v)\}   \cap \mathcal{E}(u,v)\right) \nonumber \\
&\qquad \qquad \qquad + \mathbb{P}\left(\mathcal{E}_3(u,v)^c \cap \mathcal{E}_2 \right) \nonumber \\
&\qquad \qquad \qquad+ \mathbb{P}\left(\mathcal{E}_4(u,v)^c \cap \mathcal{E}_1 \cap \mathcal{E}_2 \cap \{\deg(u) = \deg(v)\}\right) \Bigg]  \nonumber\\
&= o(1). \label{eq:final-4}
\end{align}
Here, \eqref{eq:final-0} uses \eqref{eq:equal-degrees}. Inequality \eqref{eq:final-1} follows from the identity $\mathbb{P}(A) = \mathbb{P}(A \cap B) + \mathbb{P}(A \cap B^c)$. Inequality \eqref{eq:final-2} follows from $\mathbb{P}(\mathcal{E}_1), \mathbb{P}(\mathcal{E}_2) = 1-o(1)$, as well as the union bound. Inequality \eqref{eq:final-3} uses
\[\mathcal{E}_1 \cap \mathcal{E}_2 = \mathcal{E}(u,v) ~\cup~ (\mathcal{E}_3(u,v)^c \cap \mathcal{E}_1 \cap \mathcal{E}_2) ~\cup~ (\mathcal{E}_4(u,v)^c \cap \mathcal{E}_1 \cap \mathcal{E}_2).\]
Finally, \eqref{eq:final-4} follows from \eqref{eq:sig-bound}, \eqref{eq:balanced-colors}, and Lemma \ref{lemma:H-connectivity}.
\end{proof}

\subsection{Proofs of supporting results}\label{sec:supporting}

To prove Lemma \ref{lemma:balanced-colors}, we first show that $\mathbb{E}\left[\left|\mathcal{C}_G(k) \right| \right] = (1+o(1))n/m$ for each $k \in \{0,1,\dots, m-1\}$.
\begin{lemma}\label{lemma:probability-ratio}
For a positive integer $r$ (possibly depending on $n$),
suppose that $X \sim \mathrm{Bin}(r, p_n)$. 
Let $x = rp_n( 1 \pm o(1))$ and $k = o (\sqrt{r p_n})$.
If $p_n \to 0$ and $rp_n \to \infty$, then
$$
 \frac{ \mathbb{P}( X = x + k) }{ \mathbb{P}(X = x) } = 1 \pm o(1).
$$
\end{lemma}

\begin{proof}
We have that
\begin{align}
\log \frac{ \mathbb{P}(X = x + k) }{ \mathbb{P}(X = x) } & = \log \left( \frac{ \binom{r}{x + k} p_n^{x + k} (1 - p_n)^{r - x - k} }{ \binom{r}{x} p_n^x ( 1- p_n)^{r - x} } \right) \nonumber \\
& = \log \left( \frac{ x! (r - x)! }{ (x + k)! (r - x - k)!} \left( \frac{p_n}{1 - p_n} \right)^k \right) \nonumber \\
& = \log \left( \frac{ (r-x)\ldots (r-x-k +1) }{(x + 1) \ldots (x + k)} \left( \frac{p_n}{1 - p_n} \right)^k \right) \nonumber \\
\label{eq:binomial_log_ratio}
& = \sum_{i = 0}^{k-1} \log \left( \frac{p_n}{1 - p_n} \cdot \frac{ r - x - i}{x + i + 1} \right).
\end{align}
To handle the summation in \eqref{eq:binomial_log_ratio}, we have the following asymptotic equivalence for each of the summands:
\begin{equation}
\label{eq:summand_log_ratio}
\log \left( \frac{p_n}{1 - p_n} \cdot \frac{  r - x - i}{x + i + 1} \right) \stackrel{(a)}{\sim} \log \left( \frac{1 - i/(r - x)}{1 + (i + 1)/ x} \right)  \stackrel{(b)}{\sim} - \frac{i}{r -x} - \frac{i + 1}{x}.
\end{equation}
Above, $(a)$ uses $x / (r - x) \sim p_n / (1 - p_n)$
and
$(b)$ follows since $i / (r - x) \le k / (r - x) = o(1)$ and $(i + 1) / x \le k / x = o(1)$.
Substituting \ref{eq:summand_log_ratio} into \eqref{eq:binomial_log_ratio}, we find that
\[
\left| \log \frac{ \mathbb{P} ( X = x + k) }{\mathbb{P} ( X = x)} \right| \sim \sum_{i = 0}^{k - 1} \left( \frac{i}{r - x} + \frac{i + 1}{x} \right) \le \frac{k^2}{r - x} + \frac{k^2}{x} = o(1),
\]
which is equivalent to the desired result.
\end{proof}

\begin{lemma}\label{lemma:balanced-colors-expectation}
Suppose that $p_n \to 0$ and $n p_n \to \infty$. Let $G \sim G(n, p_n)$. For any fixed positive integer $m$, we have that $\E [ | \cC_G(m, k) | ] = (1 \pm o(1)) n / m $ for all $k \in \{0, \ldots, m-1\}$. 
\end{lemma}

\begin{proof}
Let $I$ be an interval such that 
$$
[ (n-1) p_n - (np_n)^{2/3}, (n-1)p_n + (np_n)^{2/3}] \subseteq I \subseteq [(n-1)p_n - 2(np_n)^{2/3}, (n-1)p_n + 2(np_n)^{2/3}]
$$
and so that the smallest element of $\mathbb{Z} \cap I$ is a multiple of $m$ and the largest element is one less than a multiple of $m$. A useful consequence of this construction is that whenever $mk + i \in I$ for some $k \in \mathbb{Z}$ and $i \in \{0, \ldots, m-1 \}$, it also holds that $mk \in I$. 
Recalling that for any vertex $v \in [n]$ the degree satisfies $\deg(v) \sim \mathrm{Bin}(n-1, p_n)$, we have that for any $i \in \{0, \ldots, m-1 \}$, 
\begin{align}
\mathbb{P}(\deg(v) \in I, v \in \cC_G(m,i) ) & = \sum_{k \in \mathbb{Z} : mk + i \in I} \mathbb{P}(\deg(v) = mk + i) \nonumber \\
& = \sum_{k : mk \in I}  \mathbb{P}( \deg(v) = mk) \cdot \frac{ \mathbb{P} ( \deg(v) = mk + i) }{ \mathbb{P} ( \deg(v) = mk )} \nonumber  \\
& = (1 \pm o(1)) \sum_{k \in \mathbb{Z} : mk \in I} \mathbb{P}(\deg(v) = mk ) \nonumber \\
\label{eq:probability_C_m_i}
& = (1 \pm o(1)) \mathbb{P}( \deg(v) \in I , v \in  \cC_G(m,0) ).
\end{align}
Above, the equality on the second line uses that $mk + i \in I \Rightarrow mk \in I$. The inequality in the third line uses Lemma \ref{lemma:probability-ratio}, replacing the probability ratio in the second line with $1 \pm o(1)$, since $mk = np_n (1 \pm o(1))$ and $i$ is fixed with respect to $n$. Next, since $\{ \cC_G(m,i) \}_{0 \le i \le m - 1}$ partitions all vertex degrees, \eqref{eq:probability_C_m_i} implies that for any $j \in \{0, \ldots, m-1\}$, 
$$
\mathbb{P}( \deg(v) \in I) = \sum_{i = 0}^{m-1} \mathbb{P}( \deg(v) \in I , v \in  \cC_G(m,i) ) = (1 \pm o(1)) m \mathbb{P}( \deg(v) \in I , v \in  \cC_G(m,j) ).
$$
Rearranging, we see that
$$
\mathbb{P} ( \deg(v) \in I , v \in  \cC_G(m,j) ) = \frac{( 1 \pm o(1)) \mathbb{P}( \deg(v) \in I) }{ m} .
$$
It essentially remains to study $\mathbb{P}(\deg(v) \in I)$. By Chebyshev's inequality, we have that
$$
\mathbb{P} ( \deg(v) \notin I ) \le \mathbb{P} \left( | \deg(v) - (n-1) p_n | > (np_n)^{2/3} \right) \le \frac{ (n-1) p_n }{( np_n)^{4/3}} = o(1). 
$$
The above implies in particular that $\mathbb{P}( \deg(v) \in I , v \in  \cC_G(m,j) ) = (1 \pm o(1)) / m$. Furthermore, we have that
$$
\mathbb{P} ( v \in \cC_G(m,j) ) \le \mathbb{P}( \deg(v) \in I , v \in \cC_G(m,j) ) + \mathbb{P} ( \deg(v) \notin I ) = \frac{1 \pm o(1) }{m} + o(1) = \frac{1 \pm o(1) }{m}. 
$$
The desired result follows from the relation $\E [ | \cC_G(m,j) | ] = \sum_{v \in V(G) } \mathbb{P} ( v \in \cC_G(m,j) )$. 
\end{proof}

Next, we show that $\mathcal{C}_G(k)$ concentrates around its mean using a concentration inequality for Lipschitz functions of i.i.d. Bernoulli random variables. The version stated here is a special case of Corollary 1.4 in \cite{warnke_2016}.

\begin{lemma}
\label{lemma:lipschitz_concentration}
Let $Y = (Y_1, \ldots, Y_N)$ be a family of i.i.d. $\mathrm{Bern}(p)$ random variables. Assume that the function $f : \{ 0,1 \}^N \to \R$ satisfies $|f(y) - f(y') | \le C$, for any $y, y' \in \{0,1 \}^N$ which differ in a single coordinate. Then for all $t \ge 0$ we have
\[
\mathbb{P} \left( f(Y) \ge \E [ f(Y)] + t \right) \le \exp \left( - \frac{t^2 }{2C^2 p(1- p)N + 2Ct/3} \right).
\]
\end{lemma}

We are now ready to prove Lemma \ref{lemma:balanced-colors}. 

\begin{proof}[Proof of Lemma \ref{lemma:balanced-colors}]
For distinct vertices $i,j$, let $Y_{ij}$ be the indicator variable for the edge $(i,j)$. Notice that $| \mathcal{C}_G(k) |$ is a 2-Lipschitz function of the variables $\{ Y_{ij} \}_{1 \le i < j \le n}$.
Applying Lemma \ref{lemma:lipschitz_concentration} shows that
\[
\mathbb{P} \left( | \mathcal{C}_G(k) | \ge \E [ | \mathcal{C}_G(k)| ] + \frac{\epsilon n}{m} \right) \le \exp \left( - \frac{ \epsilon^2 n^2 / m^2 }{8 p_n (1 - p_n) {n \choose 2} + \frac{4 n\epsilon}{3m}} \right) \le \exp \left( - \frac{ \epsilon^2 n^2 / m^2}{5 n^2 p_n} \right) = \exp \left( - \frac{\epsilon^2}{5m^2 p_n} \right).
\]
Above, the second inequality uses that $8 p_n(1 - p_n) {n \choose 2} \le 4 n^2 p_n$ and that $4n \epsilon/ (3m) \le n^2 p_n$ since $np_n \to \infty$. Noting that the same probability bound holds when Lemma \ref{lemma:lipschitz_concentration} is applied to $-| \mathcal{C}_G(k) |$, we arrive at the bound
\[
\mathbb{P} \left( | | \mathcal{C}_G(k) | - \E [ | \mathcal{C}_G(k)|] | \ge \frac{\epsilon n}{m} \right) \le 2 \exp \left( - \frac{\epsilon^2}{5m^2 p_n} \right).
\]
Next, taking a union bound over $k \in \{0, \ldots, m - 1\}$, we arrive at
\begin{equation}
\label{eq:color_set_concentration}
\mathbb{P} \left( \exists k \in \{0, \ldots, m - 1 \} : | | \mathcal{C}_G(k) | - \E [ | \mathcal{C}_G(k)|] | \ge \frac{\epsilon n}{m} \right) \le 2m \exp \left( - \frac{\epsilon^2}{5m^2 p_n} \right).
\end{equation}
Since $m$ is a fixed positive integer, taking $p_n = o ( 1 / \log (n) )$ shows that the bound in \eqref{eq:color_set_concentration} is $o(n^{-2})$, as claimed.

To prove \eqref{eq:balanced-colors}, observe that under $\mathcal{E}_2(u) \cap \mathcal{E}_2(v)$, we have $|S^c| = (1-o(1)) n$. Therefore, the first statement implies
\begin{equation}
  \mathbb{P}(\mathcal{E}_3(u,v)^c \cap \mathcal{E}_2) \leq \mathbb{P}(\mathcal{E}_3(u,v)^c \cap \mathcal{E}_2(u) \cap \mathcal{E}_2(v))
  = o(n^{-2}).
\end{equation}
\end{proof}

\begin{remark}
In a previous version of the paper~\cite{GRS22b}, Lemma~\ref{lemma:balanced-colors} was proved by applying Freedman's inequality~\cite{freedman} to an ``edge exposure martingale'' of the type utilized by Czajka and Pandurangan~\cite{Czajka2008}. Compared to the proof in~\cite{GRS22b}, the current one, which uses Lemma~\ref{lemma:lipschitz_concentration}, is simpler and cleaner. 
We thank Lutz Warnke for referring us to~\cite{warnke_2016} and pointing out how Lemma~\ref{lemma:lipschitz_concentration} applies here. 
\end{remark}

\begin{proof}[Proof of Lemma \ref{lemma:H-connectivity}]
Observe that under $\mathcal{E}_1$, any vertex $w \in S^c$ has at most two neighbors among $\cN(u)$ and at most two neighbors among $\cN(v)$. Therefore, $\mathcal{E}_1$ ensures the first requirement of $\mathcal{E}_4(u,v)$.

It remains to control the number of vertices $w \in S^c$ which have two or more neighbors among $\cN(u) \cup \cN(v)$. Condition on a realization of $S = S(u,v)$ satisfying $\mathcal{E}_2(u) \cap \mathcal{E}_2(v) \cap \{\deg(u) = \deg(v)\}$. In particular, $|\cN(u) \cup \cN(v)| \leq 2 e np_n$. For a vertex $w \in S^c$, let $X_w = |\{\cN(w) \cap (\cN(u) \cup \cN(v))\}|$ be the number of neighbors that $w$ has among $\cN(u) \cup \cN(v)$. Note that $X_w \sim \text{Bin}(|\cN(u) \cup \cN(v)|, p_n)$. Let $\mu = |\cN(u) \cup \cN(v)| \cdot p_n \leq 2 e n p_n^2$. By Corollary \ref{corollary:Chernoff},
\[\mathbb{P}(X_i \geq 2) \leq \left(\frac{e\mu}{2} \right)^2 \leq (e^2 np_n)^2.\]
Let $Y_i = \mathbbm{1}\{X_i \geq 2\}$ and $Y = \sum_{i \in [n] \setminus S} Y_i$. 
Let $\overline{Y} \sim \text{Bin}\left(n, \left(e^2 np_n \right)^2\right)$ 
and observe that $Y$ is stochastically dominated by $\overline{Y}$. Again using Corollary~\ref{corollary:Chernoff}, we see that
\begin{align*}
\mathbb{P}(Y \geq 6) &\leq \mathbb{P}(\overline{Y} \geq 6) \leq  \left(\frac{e \cdot n \left(e^2 np_n^2 \right)^2}{6} \right)^6.
\end{align*}

It follows that 
\begin{align*}
\mathbb{P}\left(\mathcal{E}_4(u,v)^c \cap \mathcal{E}_1 \cap \mathcal{E}_2  \right) &\leq  \left(\frac{e \cdot n \left(e^2 np_n^2 \right)^2}{6} \right)^6 = o(n^{-2}). 
\end{align*}
where we have used $p_n = o(n^{-5/6})$ in the last step.
\end{proof}

\subsection{Random graph isomorphism}
In order to prove Corollary~\ref{corollary:isomorphism}, we propose a natural algorithm for finding the isomorphism of a pair of isomorphic graphs $(G_1, G_2)$. If $p_n$ is such that the degree sequences are unique with high probability, then we compute signatures according to degree sequences. Otherwise, we use our depth-$3$ signatures.
\begin{enumerate}
    \item Assign signatures to the vertices in $G_1$ and $G_2$.
    \item Sort each of the $2n$ signatures in lexicographic order.
    \item Sort the vertices in $G_1$ according to the sorted order of the signatures. Similarly, sort the vertices in $G_2$.
    \item Match the vertices in $G_1$ and $G_2$ according to this sorted order.
\end{enumerate}

\begin{proof}[Proof of Corollary \ref{corollary:isomorphism}]
Clearly this procedure succeeds if the labels are unique, which happens with high probability by Theorem~\ref{theorem:unique-signatures}. 

To analyze the runtime, observe that Step 1 runs in time $O(n\Delta)$ for depth-$2$ signatures, given the adjacency list representation. For depth-$3$ signatures, Step 1 runs in time $O(n\Delta \log n)$ as discussed above. Since the depth-$2$ signatures are of length at most $\Delta$ and depth-$3$ signatures are of length at most $\Delta m$ where $m$ is a constant, the remaining steps can be analyzed identically for the two signatures.

Step 2 requires $O\left(n \Delta \log\Delta\right)$ time, since each signature contains at most $\Delta$ elements, and can therefore be sorted in $O\left( \Delta \log \Delta\right)$ time. Step 3 (on $G_1$) requires sorting $n$ elements, with a comparison cost of $O(\Delta)$. Therefore, Step 3 takes $O\left(n \Delta \log n \right)$ time. Finally, Step 4 takes $O(n)$ time. Since these are sequential steps, the overall runtime is~$O(n \Delta \log n)$. 
\end{proof}

\section{Isomorphic 2-neighborhoods: Proof of Theorem~\ref{thm:2nbrhood_impossibility}}\label{sec:2nbrhood}

In formalizing the argument given in Section \ref{sec:outline-2-neighborhoods}, we start with a lemma that shows that most $2$-neighborhoods are trees; this is relevant for subsequently bounding the size of $B$. 

\begin{lemma}
\label{lemma:tree_like}
Let $G \sim G(n,p_{n})$ and assume that $np_n \le \log^{2} (n)$. Then, with high probability, there are at most $\log^{15}(n)$ vertices in~$G$ whose $2$-neighborhood is not a tree. 
\end{lemma}

\begin{proof}
Having at most $\log^{15}(n)$ vertices whose $2$-neighborhood is not a tree is a monotone decreasing graph property. 
Thus it suffices to prove the claim when $n p_{n} = \log^{2}(n)$, which we assume in the~following. 

Let $i \in [n]$. We start by claiming that if $\cN_2(i)$ is \emph{not} a tree, then it must contain a cycle of length at most~$5$. To see why, let $H_i$ be a tree constructed in the following manner: for all $u \in \cN_2(i)$ at distance 1 from~$i$, add the edge $(u,i)$ to the tree, and for all $v \in \cN_2(i)$ at distance 2 from $i$, choose an arbitrary neighbor $v'$ of $v$ such that $v'$ is at distance 1 from $i$ and add the edge $(v,v')$ to the tree. If $\cN_2(i)$ is not a tree, then there must exist an edge $(a,b)$ in $\cN_2(i)$ but not in $H_i$. However, since there is a path of length at most $2$ from $i$ to $a$ in $H_{i}$, and similarly a path of length at most $2$ from $i$ to $b$ in $H_{i}$, the addition of the edge $(a,b)$ creates a cycle of length at most $5$, proving the claim. 

In light of the claim, it suffices to bound the number of $2$-neighborhoods that contain a $3$-cycle, $4$-cycle, or $5$-cycle. To do so formally, we introduce a bit of notation. 
For $k \in \{3,4,5 \}$, let $\mathcal{O}_k$ be the (random) set of $k$-cycles in $G$. For a given labeled cycle $C \in \cO_k$, let $X_C$ be the number of (labeled) $2$-neighborhoods that contain $C$ as a subgraph. Define the event 
$
\cE_{\deg} : = \left \{ \max_{i \in [n]} \deg(i) \le e np_n \right \}
$
and note that $\cE_{\deg}$ holds with high probability (see Lemma~\ref{lemma:extreme-degrees}). 
We claim that 
$X_{C}$ is bounded from above by $100 \log^{4} (n)$ on the event~$\cE_{\deg}$. To see why, notice that for any cycle to be a subgraph of $\cN_2(i)$, it must be the case that $i$ is in the $2$-neighborhood of one (in fact, all) of the vertices in the cycle. Formally, we have that
\begin{equation}
\label{eq:X_C_bound}
X_{C} \le \left| \bigcup_{v \in C} \cN_2(v) \right| \le \sum_{v \in C} | \cN_2(v) |.
\end{equation}
On the event $\cE_{\deg}$, we can bound the size of any $2$-neighborhood from above by 
$1 + enp_n + (e np_n)^2 \le 20 (np_n)^2$. 
Together with \eqref{eq:X_C_bound}, this shows that 
$X_C \le 20 |C| (np_n)^2 \leq 100 \log^{4} (n)$, 
where we used that $|C| \in \{3,4,5\}$ and also that $n p_{n} \leq \log^{2}(n)$. 

Let $T$ be the set of vertices whose $2$-neighborhood is not a tree. 
By the discussion above we have that
\begin{align*}
\E [ |T| \mathbf{1}( \cE_{\deg} )] & \le \E \left[ \sum_{i \in [n]} \mathbf{1}( \cN_2(i) \text{ contains a cycle of length at most 5} ) \mathbf{1}( \cE_{\deg} ) \right] \\
& \le \E \left[ \sum_{i \in [n]} \sum_{k \in \{3,4,5 \}} \sum_{C \in \cO_k} \mathbf{1}( \cN_2(i) \text{ contains $C$} ) \mathbf{1}( \cE_{\deg} ) \right] \\
& = \E \left[ \sum_{k \in \{3,4,5\}} \sum_{C \in \cO_k} X_C \mathbf{1}(\cE_{\deg}) \right] 
\le 100 \log^{4}(n) \sum_{k \in \{3,4,5\}} \E[ | \cO_k | ], 
\end{align*}
where in the last line we first interchanged the order of summations and then used that 
$X_{C} \mathbf{1}(\cE_{\deg}) \leq 100 \log^{4}(n)$ 
from the previous paragraph. 
By a simple bound on the expected number of $k$-cycles, we have that 
\[
\sum_{k \in \{3,4,5\}} \E[ | \cO_k | ] 
\leq \sum_{k \in \{3,4,5\}} n^{k} p_{n}^{k} 
\leq 3 (n p_{n})^{5} 
\leq 3 \log^{10}(n).
\]
Putting together the previous two displays we thus have that 
$\E[ |T| \mathbf{1}(\cE_{\deg})] \leq 300 \log^{14}(n)$. 
To conclude, note that 
\[
\mathbb{P} \left( |T| \geq \log^{15}(n) \right) 
\leq \mathbb{P} \left( |T| \mathbf{1}(\cE_{\deg}) \geq \log^{15}(n) \right) + \mathbb{P} \left( \cE_{\deg}^{c} \right). 
\]
Here the first term on the right hand side is at most $300/\log(n)$ by Markov's inequality, while the second term is $o(1)$ by Lemma~\ref{lemma:extreme-degrees}. 
\end{proof}

With Lemma~\ref{lemma:tree_like} in hand, we now prove a linear lower bound on the number of good vertices. 

\begin{lemma}
\label{lemma:B_size} 
Assume that $\frac{1}{2} \log n \leq n p_{n} \leq \log^{2}(n)$. 
Then, with high probability, 
we have that $|B| \ge n/2$. 
\end{lemma}

\begin{proof}
Recall that for any $i \in [n]$, we have that $\deg(i) \sim \mathrm{Bin}(n-1,p_n)$. Bernstein's inequality for binomial random variables therefore implies that
\[
\mathbb{P} \left( i \in A \right) 
\le 2 \mathrm{exp} \left( - \frac{ (9/2) n p_{n} \log \log n }{(n-1)p_n + \sqrt{n p_{n} \log \log n }} \right) 
\le  \mathrm{exp} \left( - 4\log \log n \right),
\]
where the final inequality uses that, for $n$ sufficiently large, $(n-1)p_n$ is much larger than $\sqrt{np_n} \log \log n$ for $np_n \ge \frac{1}{2}\log n$. We can now bound the expected size of $A$ as 
\begin{equation}
\label{eq:A_size}
\E [ |A|] = \sum_{i \in [n]} \mathbb{P}(i \in A) \le n \cdot \mathrm{exp} \left( - 4 \log \log n  \right). 
\end{equation}
In particular, an application of Markov's inequality shows that 
\begin{equation}
\label{eq:A_probability}
\mathbb{P} ( |A| > n/6 ) \le 6 \cdot \mathrm{exp} \left( - 4 \log \log n  \right) = o(1). 
\end{equation}
Next, define $A'$ to be the set of vertices which have at least one neighbor in $A$. For the analysis of this set, it is again useful to define the event $\cE_{\deg} : = \{ \max_{i \in [n]} \deg(i) \le e np_n \}$; recall that $\cE_{\deg}$ holds with high probability by Lemma~\ref{lemma:extreme-degrees}. We have that
\begin{equation}
\label{eq:A'_size}
\E [ |A' | \mathbf{1}(\cE_{\deg} )] 
\le \E \left[ \sum_{i \in A} \deg(i)  \mathbf{1}(\cE_{\deg} ) \right ] 
\le e n p_{n} \E [ | A |] 
\le e n \log^{2} (n) \mathrm{exp} \left( - 4 \log \log n \right) = o(n), 
\end{equation}
where we used~\eqref{eq:A_size} and that $n p_{n} \leq \log^{2}(n)$. 
By a union bound we thus have that 
\begin{equation}
\label{eq:A'_probability}
\mathbb{P} ( |A'| > n/6 ) 
\leq 
\mathbb{P}( |A'| \mathbf{1}( \cE_{\deg}) > n/6) 
+ \mathbb{P} ( \cE_{\deg}^{c}) 
= o(1),
\end{equation}
where the final bound follows since the first term on the right hand side is $o(1)$ by~\eqref{eq:A'_size} and Markov's inequality,
and the second term is $o(1)$ by Lemma~\ref{lemma:extreme-degrees} . 

Finally, let $T$ be the set of vertices whose $2$-neighborhood is not a tree. Noting that $B^c \subseteq A \cup A' \cup T$, we may combine the estimates in \eqref{eq:A_probability}, \eqref{eq:A'_probability}, as well as Lemma~\ref{lemma:tree_like} to obtain that 
\[
\mathbb{P} ( |B^c| > n/2) \le \mathbb{P} ( |A| > n/6) + \mathbb{P} ( |A'| > n/6) + \mathbb{P} ( |T| > n/6) = o(1). 
\qedhere
\]
\end{proof}

Next, we prove a deterministic upper bound on the number of possible degree profiles of good vertices. 

\begin{lemma}\label{lemma:number_degree_profiles}
Assume that $n p_{n} = \omega \left( \log \log n \right)$. 
For all $n$ large enough, the number of possible degree profiles for a good vertex is at most 
$\exp \left( 4 \sqrt{n p_{n} \log \log n} \log \left( n p_{n} \right)  \right)$. 
\end{lemma}
\begin{proof}
Enumerate all integers in the interval 
$((n-1)p_n - 3\sqrt{n p_{n} \log \log n} , (n-1)p_n + 3\sqrt{n p_{n} \log \log n} )$ 
in increasing order as 
$d_{1} < d_{2} < \ldots < d_{k}$, 
and note that $k$ satisfies 
$6\sqrt{n p_{n} \log \log n}  - 1 \leq k \leq 6\sqrt{n p_{n} \log \log n}  + 1$. 

Let $v$ be a good vertex and let $d$ denote its degree. 
Since $v$ is good, we have that $v \notin A$, and so we must have that $d \in \{ d_{1}, d_{2}, \ldots, d_{k} \}$; 
in particular, 
$(n-1)p_n - 3\sqrt{n p_{n} \log \log n}  
\leq d 
\leq (n-1)p_n + 3\sqrt{n p_{n} \log \log n} $. 
Furthermore, since $v$ is good, all its neighbors are also not in $A$, 
and thus their degrees are also in the set $\{ d_{1}, d_{2}, \ldots, d_{k} \}$. 
For every $\ell \in [k]$, let $x_{\ell}$ denote the number of vertices of degree $d_{\ell}$ in the degree profile of~$v$. 
Then the collection of nonnegative integers $\{ x_{1}, x_{2}, \ldots, x_{k} \}$ must solve the equation 
$x_{1} + \ldots + x_{k} = d$. 
Through a stars-and-bars argument, the number of solutions to this equation is exactly~$\binom{d+k-1}{k-1}$. 

Next, we bound the asymptotic behavior of this quantity as $n \to \infty$. We have that 
\[
\log \binom{d+k-1}{k-1} 
\leq (k-1) \log \left( \frac{e(d+k-1)}{k-1} \right) 
\leq k \log \left( \frac{4d}{k} \right), 
\]
where the first inequality uses 
$\binom{a}{b} \leq (ea/b)^{b}$ for $a \geq b$, 
and the second inequality uses the bounds on $k$ and~$d$ (see above), 
as well as the assumption 
$n p_{n} = \omega \left( \log \log n \right)$. 
Using these same bounds we also have that~$4d/k \leq \sqrt{n p_{n}}$. 
Plugging this into the display above and using the upper bound on $k$, we obtain that 
\[
\log \binom{d+k-1}{k-1} 
\leq \left( 6 \sqrt{n p_{n} \log \log n}  + 1 \right) \log \left( \sqrt{n p_{n}} \right) 
\leq 3.1 \sqrt{n p_{n} \log \log n} \log \left( n p_{n} \right), 
\]
where the second inequality holds for all $n$ large enough. 

In summary, the number of possible degree profiles is the number of ways to choose $d$ and $\{x_{1}, \ldots, x_{k} \}$. 
The bound $d \leq (n p_{n} + 3\sqrt{n p_{n} \log \log n} )$ 
follows from the definition of a good vertex, 
and above we showed that the number of ways to choose $\{x_{1}, \ldots, x_{k} \}$ is at most 
$\exp \left( 3.1 \sqrt{n p_{n} \log \log n} \log \left( n p_{n} \right)  \right)$. 
Finally, we have that 
$(n p_{n} + 3 \sqrt{n p_{n} \log \log n} ) \exp \left( 3.1 \sqrt{n p_{n} \log \log n} \log \left( n p_{n} \right)  \right)
\leq \exp \left( 4 \sqrt{n p_{n} \log \log n} \log \left( n p_{n} \right) \right)$. 
\end{proof}

We conclude the proof by the pigeonhole principle; 
Theorem~\ref{thm:2nbrhood_impossibility} now follows directly from Lemma~\ref{lemma:pigeonhole}. 
\begin{lemma}\label{lemma:pigeonhole}
Assume that 
$\frac{1}{2} \log n \leq n p_{n} \leq c \log^{2}(n) / \left( \log \log n \right)^{3}$ for some $c \in (0,1/64)$. 
Then, with high probability, there exist at least 
$\frac{1}{2} n^{1-8\sqrt{c}}$ 
good vertices with the same degree profile. 
\end{lemma}
\begin{proof}
Let $\cE'$ denote the event that there are at least $n/2$ good vertices; by Lemma~\ref{lemma:B_size}, this event holds with high probability. 
By Lemma~\ref{lemma:number_degree_profiles}, the number of possible degree profiles for a good vertex is at most 
$\exp \left( 4 \sqrt{n p_{n} \log \log n} \log \left( n p_{n} \right)  \right) 
\leq \exp \left(8 \sqrt{c} \log n \right)
= n^{8\sqrt{c}}$, 
where the inequality uses the assumption that 
$n p_{n} \leq c \log^{2}(n) / \left( \log \log n \right)^{3}$. 
Thus, on the event $\cE'$, by the pigeonhole principle there must exist some degree profile with at least $(n/2)/n^{8\sqrt{c}} = \frac{1}{2} n^{1-8\sqrt{c}}$ good vertices having this degree profile. 
\end{proof}

\section{Smoothed analysis of graph isomorphism: Proof of Theorem \ref{theorem:smoothed-analysis}}\label{sec:smoothed-analysis}
\begin{lemma}\label{lemma:balanced-colors-semirandom}
Fix $\epsilon > 0$. Let $G_1 \in \mathcal{G}_n(\lambda)$ for some $\lambda < 1$ and let $G_2 \sim G(n,p_n)$, where $p_n = \omega (\log^2(n) / n)$ and $p_n = o ( 1/ \log^3 (n) )$.
Let $G$ be the union of $G_1$ and $G_2$, and let $m = O(\log n)$. Then
\[
\mathbb{P} \left( \forall k \in \{ 0, \ldots, m - 1 \}, (1 - \epsilon) \frac{n}{m} \le | \mathcal{C}_G(m,k) | \le (1 + \epsilon) \frac{n}{m} \right) = 1 - o(n^{-3}).
\]
Moreover, the same result holds if $G$ is the \xor of $G_1$ and $G_2$.
\end{lemma}

The proof of Theorem \ref{theorem:smoothed-analysis} requires that for every $u \neq v$, there be a ``witness'' vertex $w \in \cN(u) \setminus ( \cN(v) \cup \{v \})$. The following lemma ensures that such a witness vertex exists.
\begin{lemma}
\label{lemma:neighborhood_difference}
For $\lambda \in (0,1)$, let $G$ be the union of $G_1 \in \cG(\lambda)$ and $G_2 \sim G(n,p_n)$, where $p_n \to 0$ and $p_n = \omega(\log(n) / n)$.
Then for any $u,v \in [n]$, it holds that $\cN(u) \setminus ( \cN(v) \cup \{v \}) \neq \emptyset$ with probability $1 - o(n^{-2})$.

Furthermore, the same result holds if $G$ is the \xor of $G_1$ and $G_2$.
\end{lemma}

\begin{proof}
Define $S : = \{ u, v \} \cup \cN_{G_1}(u) \cup \cN_{G_1}(v)$, and notice that $|S| = O( n^\lambda )$ since $G_1 \in \cG(\lambda)$.
Furthermore, for any $w \notin S$, it holds 
with probability $p_n (1 - p_n)$ that $w \in \cN_{G_2}(u) \setminus (\cN_{G_2}(v) \cup S)$.
As this probability is independent over all $w \notin S$, we have that
\begin{align*}
\mathbb{P} \left( \cN_{G_2}(u) \setminus (\cN_{G_2}(v) \cup S)= \emptyset \right) & = \left( 1 - p_n (1 - p_n) \right)^{n - |S|} \\
& \stackrel{(a)}{\le} \exp \left( - (n - |S|) p_n (1 - p_n) \right) \\
& \stackrel{(b)}{=} \exp \left( - (1 - o(1)) np_n \right) \\
& \stackrel{(c)}{=} o(n^{-2}).
\end{align*}
Above, $(a)$ uses the inequality $1 + x \le e^x$ for any $x \in \mathbb{R}$;
$(b)$ follows since $n - |S| \sim n$ and $p_n(1 - p_n) \sim p_n$ as $n \to \infty$;
$(c)$ uses our assumption that $p_n = \omega (\log (n) / n)$.

Finally, we note that if $\cN_{G_2}(u) \setminus ( \cN_{G_2}(v) \cup S) \neq \emptyset$, and $G$ is the union of $G_1$ and $G_2$, then $\cN(u) \setminus ( \cN(v) \cup \{v\}) \neq \emptyset$.
The same result holds when $G$ is the \xor of $G_1$ and $G_2$.
\end{proof}

We next define some useful events for our analysis. 
The first event provides an upper bound for vertex degrees in $G_2$:
\[
\mathcal{F}_1 : = \left \{ \forall z \in [n], \deg_{G_2}(z) \le e np_n \right \}.
\]
Since $np_n \ge 4 \log n$, Lemma \ref{lemma:extreme-degrees} implies that $\mathbb{P}( \mathcal{F}_1 ) = 1 - o(n^{-2})$.

For fixed vertices $u,v$, we also define the event 
\[
\mathcal{F}_2(u,v) : = \{ \cN(u) \setminus ( \cN(v) \cup \{v \} ) \neq \emptyset \}.
\]
By Lemma \ref{lemma:neighborhood_difference}, we have that $\mathbb{P}( \mathcal{F}_2) = 1 - o(n^{-2})$.

Finally, we introduce an event concerning the size of color classes in $G$. 
For a fixed $\epsilon > 0$, define the event
\[
\mathcal{F}_3 : = \left \{ \forall k \in \{0, \ldots, m - 1\},  | \mathcal{C}_{G} (m,k) |  \ge (1 - \epsilon) \frac{n}{m} \right \}.
\]
In light of Lemma \ref{lemma:balanced-colors-semirandom}, $\mathbb{P}( \mathcal{F}_3) = 1-o(n^{-2})$.
Together, we have that $\mathbb{P} ( \mathcal{F}_1 \cap \mathcal{F}_2(u,v) \cap \mathcal{F}_3) = 1 - o(n^{-2})$.

Finally, we provide a probabilistic lower bound on the cardinality of the set $R$, as described by in Section~\ref{sec:smoothed-analysis-outline}.

\begin{lemma}\label{lemma:size-R}
Fix $u \neq v$. On the event $\mathcal{F}_1 \cap \mathcal{F}_2(u,v)$, it holds that $|R(u,v)| \geq n - n^{\lambda + o(1)}$.
\end{lemma}
\begin{proof}
Assume that that the event $\mathcal{F}_1 \cap \mathcal{F}_2(u,v)$ holds, so that $\mathcal{N}(u) \setminus (\mathcal{N}(v) \cup \{v\})$ is nonempty and the maximum degree in $G_2$ is $e np_n$. We lower-bound $|R|$, by upper-bounding the number of pairs $(w,w')$ failing to satisfy the properties of $R$. 
In particular, observe that:
\begin{enumerate}

    \item If $G_1 \in \mathcal{G}(\lambda)$, then $| \mathcal{N}(u) | \le n^\lambda + e np_n$; the same bound holds for $| \mathcal{N}(v) |$. As a result, there are at most $2+ 2n^{\lambda} + 2enp_n$ vertices which fail to satisfy Property \ref{item:R_neighborhood}.
    
    \item By accounting for the contributions from $G_1$ and $G_2$, the neighborhood of $\mathcal{N}(u) \cup \mathcal{N}(v) \setminus \{w\}$ (which are precisely the vertices that fail to satisfy Property \ref{item:R_2_neighborhood}) is of size at most
\[2\left(n^{\lambda} + n^{\lambda} enp_n + enp_n (enp_n + n^{\lambda}) \right).\]
Here, the factor of $2$ is due to separating the contributions of $u$ and $v$. Focusing on the contributions of $u$, the term $n^{\lambda}$ comes from the distance-$2$ neighbors of $u$ in $G_1$, since $G_1 \in \mathcal{G}(\lambda)$. The term $n^{\lambda} enp_n$ comes from the random neighbors of the neighbors of $u$ in $G_1$. Finally, the term $enp_n (enp_n + n^{\lambda})$ comes from the deterministic and random neighbors of the neighbors of $u$ in $G_2$.

\item  For $G_1 \in \mathcal{G}(\lambda)$, $w$ can have at most $n^{\lambda}$ neighbors in $G_1$, hence there are at most $n^\lambda$ vertices which fail to satisfy Property \ref{item:R_G1_edges}. 
\end{enumerate}
Putting everything together, if the event $\mathcal{F}_1 \cap \mathcal{F}_2(u,v)$ holds, then
\[|R^c| \leq 2+ 2n^{\lambda} + 2enp_n  + 2\left(n^{\lambda} + n^{\lambda} enp_n + enp_n (enp_n + n^{\lambda})\right) + n^{\lambda} = n^{\lambda + o(1)}, 
\]
for $n$ large enough. 
In the final expression, we have used that $np_n = n^{o(1)}$.
Therefore, $|R| \geq n - n^{\lambda + o(1)}$ on $\mathcal{F}_1 \cap \mathcal{F}_2(u,v)$.
\end{proof}

\begin{proof}[Proof of Theorem \ref{theorem:smoothed-analysis}, Part 1] 
Fix vertices $u,v \in [n]$, and let $w \in \mathcal{N}(u) \setminus ( \mathcal{N}(v) \cup \{v \} )$.
Our strategy is to show that for any $x \in \mathcal{N}(v)$, it holds with high probability that $\ell_G(x) \neq \ell_G(w)$. This in turn implies that $\text{sig}(u) \neq \text{sig}(v)$. We will then take a union bound over pairs of vertices $u,v$ to prove the desired result.

To formalize this strategy, take $w \in \mathcal{N}(u) \setminus ( \mathcal{N}(v) \cup \{v \})$, which is possible on the event $\mathcal{F}_2(u,v)$.
Let us also define $H$ to be the graph consisting of all edges in $G$ revealed by $\mathcal{I}$; that is, $H$ contains all edges in $G$ except for the edges $\{ (w, w') : w' \in R(u,v) \}$. A useful observation is that the presence or absence of these edges affect only the colors of vertices in $\{w \} \cup (\mathcal{N}(w) \cap R)$.
As a result, it holds for all $k \in \{0, \ldots, m - 1\}$ that $| |\mathcal{C}_G(m,k)| - | \mathcal{C}_H(m,k) | | \le | \{ w \} \cup \mathcal{N}(w) |$.
With this inequality in hand, we have the following lower bound for the size of color classes in $H$ for each $k \in \{0, \ldots, m - 1 \}$ on the event $\mathcal{F}_1 \cap \mathcal{F}_2(u,v) \cap \mathcal{F}_3$:
\begin{align*}
| \mathcal{C}_H(m,k) \cap R| & \ge | \mathcal{C}_H(m,k) | - |R^c| \\
& \ge | \mathcal{C}_G(m,k) | - | \{ w \} \cup \mathcal{N}(w) | - |R^c| \\
& \stackrel{(a)}{\ge} ( 1 - \epsilon) \frac{n}{m} - (1 + n^\lambda + enp_n) - n^{\lambda + o(1)} \\
& \stackrel{(b)}{\ge} (1 - 2 \epsilon) \frac{n}{m}.
\end{align*}
Above, $(a)$ follows from the lower bound on $| \mathcal{C}_G(m,k)|$ from $\mathcal{F}_3$, the upper bound $| \mathcal{N}(w) | \le | \mathcal{N}_{G_1}(w) | + | \mathcal{N}_{G_2}(w) | \le n^\lambda + enp_n$ on $\mathcal{F}_1(u,v)$, and the upper bound on $|R^c|$ on $\mathcal{F}_1 \cap \mathcal{F}_2(u,v)$ (see Lemma \ref{lemma:size-R}).
The final inequality $(b)$ holds for $n$ sufficiently large, since $m = \log n$.
To summarize, we have shown that if we define the event 
\[
\mathcal{F} : = \left \{ \forall k \in \{ 0, \ldots, m - 1 \}, | \mathcal{C}_H(m,k) \cap R | \ge (1 - 2 \epsilon) \frac{n}{m} \right \},
\]
then $\mathcal{F}_1 \cup \mathcal{F}_2(u,v) \cap \mathcal{F}_3 \subset \mathcal{F}$, and
\begin{equation}
\label{eq:F_probability}
\mathbb{P} ( \mathcal{F}^c) \le \mathbb{P}( \mathcal{F}_1^c) + \mathbb{P}( \mathcal{F}_2(u,v)^c) + \mathbb{P}( \mathcal{F}_3^c) = o(n^{-2}).
\end{equation}
A useful fact about $\mathcal{F}$ is that it depends only on $H$ and $F$, and is thus measurable with respect to the revealed information $\mathcal{I}$.

Next, we characterize the color list of $w$ in $G$.
To this end, let $X_0, \ldots, X_{m-1}$ be a collection of independent random variables, with 
$X_k \sim \mathrm{Bin} ( | \mathcal{C}_H(m,k) \cap R |, p_n)$.
In particular, $X_k$ is equal in distribution to the number of vertices in $\mathcal{N}(w) \cap R$ which are in the color class $\mathcal{C}_H(m,k)$.
Since the color with respect to $G$ of a vertex in $\mathcal{N}(w) \cap R$ is one more than its color with respect to $H$, it holds that
\[
\ell_G(w)_{k + 1} \stackrel{d}{=} \ell_H(w)_{k+1} + X_k, \qquad k \in \{ 0, \ldots, m - 1 \}.
\]
Due to this representation, we can study the probability that $\ell_G(w)$ takes on a particular value conditioned on $\mathcal{I}$. For any fixed $\mathbf{z} \in \{ 0, \ldots, n-1\}^m$, we have that 
\begin{align}
\mathbb{P} ( \ell_G(w) = \mathbf{z} \vert \mathcal{I} ) \mathbf{1}( \mathcal{F} ) & = \mathbb{P} \left( \left. \bigcap_{k = 0}^{m - 1} \{ X_k = z_{k + 1} - \ell_H(w)_{k + 1} \} \right \vert \mathcal{I} \right) \mathbf{1}( \mathcal{F} ) \nonumber  \\
& \stackrel{(c)}{=} \left \{ \prod_{k = 0}^{m-1} \mathbb{P}( X_k = z_{k + 1} - \ell_H(w)_{k + 1} \vert \mathcal{I} ) \right \} \mathbf{1}( \mathcal{F} ) \nonumber \\
& \stackrel{(d)}{\le}  \left \{ \prod_{k = 0}^{m-1} \left( \frac{e^4}{\sqrt{ | \mathcal{C}_H(m,k) \cap R | p_n}} \right) \right \} \mathbf{1}( \mathcal{F} ) \nonumber \\
\label{eq:color_list_probability_bound}
 & \stackrel{(e)}{\le} \left( \frac{ e^8 m}{ (1 - 2 \epsilon) np_n} \right)^{m/2},
\end{align}
where $(c)$ follows from the conditional independence of the $X_k$'s with respect to $\mathcal{I}$,
$(d)$ uses Lemma \ref{lemma:binomial-upper-bound} to bound the probability that a binomial random variable takes on a particular deterministic value, 
and $(e)$ uses the lower bound on $| \mathcal{C}_H(m,k) \cap R|$ given by $\mathcal{F}$.

Next, we study the probability that $\ell_G(w)$ is equal to $\ell_G(x)$, for some neighbor $x$ of $v$.
Since $x$ has no neighbors in $R$ by construction, the only potential neighbor of $x$ that could have a different color in $H$ and $G$ is $w$.
As a result, it holds that either $\ell_G(x) = \ell_H(x)$, or $\ell_G(x)$ and $\ell_H(x)$ differ by a unit transposition (i.e., the contribution of $w$ to $x$'s color list changes from one entry of the color list to another).
The number of potential color lists corresponding to $x$ in $G$, conditioned on $\mathcal{I}$, is therefore at most $1 + m(m-1) \le m^2$.
Taking a union bound over these possibilities and using the probability bound \eqref{eq:color_list_probability_bound} shows that
\begin{equation}
\label{eq:color_list_equality}
\mathbb{P} ( \ell_G(w) = \ell_G(x) \vert \mathcal{I} ) \mathbf{1}( \mathcal{F} ) \le m^2 \left( \frac{ e^8 m}{(1 - 2 \epsilon)np_n} \right)^{m/2}.
\end{equation}
Putting everything together, we can bound the probability of interest as follows:
\begin{align*}
\mathbb{P} ( \mathcal{F} \cap \{ \text{sig}(u) = \text{sig}(v) ) & \le \mathbb{P}  ( \mathcal{F} \cap \{ \forall w \in \mathcal{N}(u) , \exists x \in \mathcal{N}(v) : \ell_G(w) = \ell_G(x) \} ) \\
&  \stackrel{(f)}{\le} \E \left[ \mathbb{P} ( \exists x \in \mathcal{N}(v) : \ell_G(w) = \ell_G(x) \vert \mathcal{I} ) \mathbf{1}( \mathcal{F} ) \right ] \\
&  \stackrel{(g)}{\le} \E \left[ | \mathcal{N}(v) | m \left( \frac{4e^8 m}{(1 - 2 \epsilon)np_n} \right)^{m/2} \right ] \\
&  \stackrel{(h)}{\le} nm \left( \frac{4e^8 m}{(1 - 2 \epsilon)np_n} \right)^{m/2} \stackrel{(i)}{=} o(n^{-2}).
\end{align*}
Above, $(f)$ follows from picking a particular $w \in \mathcal{N}(u) \setminus ( \mathcal{N}(v) \cup \{v \})$ and revealing the corresponding information $\mathcal{I}$ (such a $w$ exists on the event $\mathcal{F}$),
$(g)$ is due to a union bound over $x \in \mathcal{N}(v)$, and by using the bound \eqref{eq:color_list_equality}.
The inequality $(h)$ uses $| \mathcal{N}(v) | \le n$,
and $(i)$ follows from our choice $m = \log n$ and $np_n = \omega (\log^2 n)$.

Moreover, since $\mathbb{P}( \mathcal{F}) = 1 - o(n^{-2})$ by \eqref{eq:F_probability}, a union bound implies
\[
\mathbb{P} ( \text{sig}(u) = \text{sig}(v) ) \le \mathbb{P}( \mathcal{F}^c) + \mathbb{P} ( \mathcal{F} \cap \{ \text{sig}(u) = \text{sig}(v) \} ) = o(n^{-2}).
\]
The desired result now follows from taking a union bound over all pairs of distinct vertices $u,v$.
\end{proof}

\begin{proof}[Proof of Theorem \ref{theorem:smoothed-analysis}, Part 2]
Fixing $u \neq v$, let $w$ be an arbitrary vertex from the set $\mathcal{N}(u) \setminus \left(\mathcal{N}(v) \setminus \{v\} \right)$. Define the set $R = R(u,v)$ identically to Section \ref{sec:smoothed-analysis-outline}.
To upper-bound $|R^c|$, we couple $G'$ to $G = G_1 \cup G_2$, such that (1) any edge in $G_1$ belongs to $G$, and belongs to $G'$ with probability $1-p_n$; (2) any edge in $G_2$ but not $G_1$ belongs to both $G$ and $G'$.
Due to the monotonicity of the coupling, the same upper bound on $|R^c|$ holds. 

Let $H'$ be the graph $G'$ formed by removing the edges in the set $\{ (w,w') : w' \in R \}$.
Using similar steps as in the proof of Proof of Theorem \ref{theorem:smoothed-analysis}, Part 1, we conclude that $\left|\mathcal{C}_{H'}(m,k) \cap R\right| \geq (1 - 2 \epsilon)\frac{n}{m}$ for all $k \in \{0,1, \dots, m-1\}$, for any fixed $\epsilon > 0$, under an event which occurs with probability $1-o(n^{-2})$. Finally, since the remaining random edges $\{(w,w') : w' \in R\}$ are not edges in $G_1$, the remainder of the proof is identical.
\end{proof}

To prove Lemma \ref{lemma:balanced-colors-semirandom}, we proceed similarly as in the proof of Lemma \ref{lemma:balanced-colors}, and study the expectation of the size of each color class. 

\begin{lemma}\label{lemma:balanced-colors-expectation-semirandom}
Suppose that $G_1 \in \cG(\lambda)$ for some $\lambda \in (0,1)$, let $G_2 \sim G(n, p_n)$, and let $G$ be the union of $G_1$ and $G_2$.
Let us further assume that $np_n \to \infty$, $n^{(2 \lambda + 1) / 3} p_n \to 0$, and $m = o(\sqrt{n p_n})$.
Then $\E [ | \cC_G(m,k) | ] = (1 \pm o(1)) n / m$ for all $k \in \{0, \ldots, m - 1 \}$ and $n$ large enough.
Furthermore, the same result holds if $G$ is the $\mathrm{XOR}$ of $G_1$ and $G_2$.
\end{lemma}

\begin{proof}
Let us first consider the case where $G$ is the union of $G_1$ and $G_2$.
Fix a vertex $v \in [n]$, and define as a shorthand $d_v : = \deg_{G_1}(v)$.
Let us also define $I_v$ to be an interval such that
\[
[(n - 1 - d_v) p_n - (np_n)^{3/4}, (n - 1 - d_v) p_n + (np_n)^{3/4}] \subseteq I_v \subseteq [(n-1-d_v)p_n - 2(np_n)^{3/4}, (n - 1 - d_v)p_n + 2(np_n)^{3/4}],
\]
and so that the smallest element of $\mathbb{Z} \cap I_v$ is a multiple of $m$ and the largest element is one less than a multiple of $m$.
As in the proof of Lemma \ref{lemma:balanced-colors-expectation}, if $mk + i \in I_v$ for some $k \in \mathbb{Z}$ and $i \in \{0, \ldots, m - 1 \}$, then it also holds that $mk \in I_v$.
Next, noting that $\deg_G(v) - d_v \sim \mathrm{Bin}(n - 1 - d_v, p_n)$, we have for any $i \in \{0, \ldots, m - 1 \}$ that
\begin{align}
\mathbb{P}(\deg_G(v) - d_v \in I_v, v \in \cC_G(m,d_v + i) ) & = \sum_{k \in \mathbb{Z} : mk + i \in I_v} \mathbb{P}(\deg_G(v)-d_iv = mk + i) \nonumber \\
& = \sum_{k : mk \in I_v}  \mathbb{P}( \deg_G(v)-d_v = mk) \cdot \frac{ \mathbb{P} ( \deg_G(v)-d_v = mk + i) }{ \mathbb{P} ( \deg_G(v)-d_v = mk )} \nonumber  \\
& = (1 \pm o(1)) \sum_{k \in \mathbb{Z} : mk \in I_v} \mathbb{P}(\deg_G(v) - d_v = mk ) \nonumber \\
\label{eq:probability_C_m_i_smoothed}
& = (1 \pm o(1)) \mathbb{P}( \deg_G(v) - d_v \in I_v , v \in  \cC_G(m,d_v) ).
\end{align}
The equality on the second line uses that $mk + i \in I_v \implies mk \in I_v$ for $i \in \{0, \dots, m-1\}$. The equality in the third line follows from Lemma \ref{lemma:probability-ratio}, since $mk = (1 \pm o(1)) np_n$ and $i \le m = o ( \sqrt{n p_n})$. 
Since $\{ \cC_G(m,i) \}_{0 \le i \le m - 1}$ partitions all vertex degrees, \eqref{eq:probability_C_m_i_smoothed} implies that for any $j \in \{0, \ldots, m-1\}$, 
\[
\mathbb{P} ( \deg_G(v)-d_v \in I_v , v \in  \cC_G(m,j) ) = (1 \pm o(1)) \frac{\mathbb{P}( \deg_G(v)-d_v \in I_v) }{m}.
\]
It remains to characterize the probability in the numerator.
This can be done by Chebyshev's inequality:
\begin{equation}
\label{eq:degree_chebyshev}
\mathbb{P} ( \deg_G(v) - d_v \notin I_v ) \le \mathbb{P} ( | \deg_G(v) - (n - 1 - d_v) p_n| \ge (np_n)^{3/4} ) \le \frac{(n-1-d_v) p_n (1 - p_n)}{(np_n)^{3/2}} = o \left( \frac{1}{m} \right),
\end{equation}
which implies in particular that
\[
\mathbb{P}( v \in \cC_G(m,j) ) \ge \mathbb{P}( \deg_G(v) - d_v \in I_v, v \in \cC_G(m,j) ) = \frac{1 \pm o(1)}{m}.
\]
A matching upper bound can be obtained as follows:
\[
\mathbb{P} ( v \in \cC_G(m,j)) \le \mathbb{P} ( \deg_G(v) - d_v \in I_v, v \in \cC_G(m,j) ) + \mathbb{P} ( \deg_G(v) - d_v \notin I_v) = \frac{1 \pm o(1)}{m},
\]
where the final bound uses \eqref{eq:degree_chebyshev}.
It follows that $\E [ |\cC_G(m,j) | ] = (1 \pm o(1)) n / m$, as desired.

Next, we study the case where the graph is the $\mathrm{XOR}$ of $G_1$ and $G_2$.
We denote this graph by $G'$ to distinguish it from the union of $G_1$ and $G_2$, which was denoted by $G$ in the earlier part of the proof.
We can write $\deg_{G'}(v) = d_v - X_v + Y_v$, where $X_v \sim \mathrm{Bin}(d_v, p_n)$ denotes the number of
edges in $G_1$ incident to $v$ that become non-edges after the $\mathrm{XOR}$ with $G_2$, and $Y_v \sim \mathrm{Bin}(n - 1 - d_v, p_n)$ denotes the number of edges in $G_2$ incident to $v$ that are not already edges in $G_1$.
Notice in particular that $\mathbb{P} ( X_v \neq 0) \le d_v p_n$ by Markov's inequality, hence we can couple $\deg_{G'}(v)$ and $\deg_G(v)$ with probability $1 - d_v p_n$. Then
\begin{equation}
\label{eq:d_p_m_product}
d_v p_n m = o \left( n^{\lambda} p_n \sqrt{n p_n} \right) = o \left(  \left( n^{( 2 \lambda + 1) / 3} p_n \right)^{3/2} \right) = o(1).
\end{equation}
Above, the first equality follows since $G_1 \in \cG(\lambda)$ implies that $d_v \le n^{\lambda}$, and since $m = o ( \sqrt{n p_n})$;
the final equality is due to our assumption that $n^{(2 \lambda + 1) / 3} p_n = o(1)$.
A useful consequence of \eqref{eq:d_p_m_product} is that $d_v p_n = o(1/m)$.
Hence, for any $j \in \{0, \ldots, m-1 \}$,
\[
\left| \mathbb{P}( v \in \cC_{G'}(m,j) ) - \mathbb{P}(v \in \cC_G(m,j) ) \right| = o\left( \frac{1}{m} \right).
\]
In particular, $\mathbb{P}(v \in \cC_{G'}(m,j) ) = (1 \pm o(1)) / m$, and the desired result follows.
\end{proof}

\begin{proof}[Proof of Lemma \ref{lemma:balanced-colors-semirandom}]
The proof is nearly identical to the proof of Lemma \ref{lemma:balanced-colors} except for a few minor details, which we elaborate on below.

If $G$ is the union of $G_1$ and $G_2$, then $| \mathcal{C}_G(k) |$ is a function of the deterministic graph $G_1$ as well as the i.i.d. edge indicators of $G_2$. Moreover, since the presence or absence of the edge $(i,j)$ changes $| \mathcal{C}_G(k) |$ by at most 2, we may apply Lemma \ref{lemma:lipschitz_concentration} as in the proof of Lemma \ref{lemma:balanced-colors} to obtain the same probability bound \eqref{eq:color_set_concentration}. For convenience, we state the bound below:
\[
\mathbb{P} \left( \exists k \in \{0, \ldots, m - 1 \} : | | \mathcal{C}_G(k) | - \E [ | \mathcal{C}_G(k) |] | \ge \frac{\epsilon n}{m} \right) \le 2m \exp \left( - \frac{\epsilon^2}{5 m^2 p_n} \right).
\]
If $m = O ( \log n)$ and $p_n = o(1 / \log^3 n)$, then this bound is $o(n^{-2})$. The same arguments also hold in the case where $G$ is the XOR of $G_1$ and $G_2$.
\end{proof}

\section*{Acknowledgements}
A.S. acknowledges support from the U.S. National Science Foundation under Grants  CCF-1908308 and ECCS-2039716, as well as a grant from the C3.ai Digital Transformation Institute. We thank Lutz Warnke for observing that the method of typical bounded differences simplifies the proof of Lemma~\ref{lemma:balanced-colors} (and thus also the proof of Lemma~\ref{lemma:balanced-colors-semirandom}). We thank Aravindan Vijayaraghavan for helpful discussions, as well as the reviewers of an earlier version of this work who gave many helpful comments. 
\bibliography{references}
\bibliographystyle{plain}

\end{document}